\documentclass[11pt,reqno]{amsart}

\usepackage{amsmath,amssymb,mathrsfs}
\usepackage{graphicx,cite,cases}

\newtheorem{theo}{Theorem}[section]

\newtheorem{lemm}[theo]{Lemma}

\numberwithin{equation}{section}

\begin{document}

\title[inverse elastic surface scattering]{Inverse elastic surface scattering
with far-field data}

\author{Huai-An Diao}
\address{School of Mathematics and Statistics, Northeast Normal University,
Changchun, Jilin 130024, China.}
\email{hadiao@nenu.edu.cn}

\author{Peijun Li}
\address{Department of Mathematics, Purdue University, West Lafayette, Indiana
47907, USA.}
\email{lipeijun@math.purdue.edu}

\author{Xiaokai Yuan}
\address{Department of Mathematics, Purdue University, West Lafayette, Indiana
47907, USA.}
\email{yuan170@math.purdue.edu}

\thanks{The research of H.-A.  Diao was supported in part by the Fundamental
Research Funds for the Central Universities under the grant 2412017FZ007. The
research of P. Li was supported in part by the NSF grant DMS-1151308.}



\begin{abstract}
A rigorous mathematical model and an efficient computational method are proposed
to solving the inverse elastic surface scattering problem which arises from the
near-field imaging of periodic structures. We demonstrate how an enhanced
resolution can be achieved by using more easily measurable far-field data. The
surface is assumed to be a small and smooth perturbation of an elastically rigid
plane. By placing a rectangular slab of a homogeneous and isotropic elastic
medium with larger mass density above the surface, more propagating wave modes
can be utilized from the far-field data which contributes to the reconstruction
resolution. Requiring only a single illumination, the method begins with the
far-to-near (FtN) field data conversion and utilizes the transformed field
expansion to derive an analytic solution for the direct problem, which leads to
an explicit inversion formula for the inverse problem. Moreover, a nonlinear
correction scheme is developed to improve the accuracy of the reconstruction.
Results show that the proposed method is capable of stably reconstructing
surfaces with resolution controlled by the slab's density.
\end{abstract}

\maketitle

\section{Introduction}

Scattering problems have been studied extensively in the past decades
\cite{coltonbook13}. They have many significant applications in many science and
engineering areas such as radar and sonar, medical imaging, and remote sensing.
Especially, the elastic wave scattering problems have practical applications
in geophysics, seismology, and nondestructive testing
\cite{carlosammari01,ak04,aknt02,bc05ip}. There are two kinds of problems: the
direct scattering problems are to determine the wave field from the
differential equations governing the wave motion; the inverse scattering
problems are to determine the unknown medium, such as the
geometry or material, from the measurement of the wave field. In this paper we
focus on the inverse elastic scattering problem in periodic structures. The
direct elastic scattering problem has been studied by many researchers
\cite{arens99, arens99integral, eh10mm, eh12m3}. The uniqueness result of the
inverse problem can be found in \cite{cgk02ip}. The numerical study can be
found in \cite{eh12cp} and \cite{hlz13ip} for the inverse problem by using an
optimization method and the factorization method, respectively.

It is known that there is a resolution limit to the sharpness of the details
which can be observed from conventional far-field optical microscopy, one half
the wavelength, referred to as the Rayleigh criterion or the diffraction
limit \cite{Courjon03}. The loss of resolution is mainly due to the ignorance of
the evanescent wave components. Near-field optical imaging is an effective
approach to obtain images with subwavelength resolution. The inverse
scattering problems via the near-field imaging for acoustic and electromagnetic
waves have been undergoing extensive studies for impenetrable infinite rough
surfaces \cite{bl13am}, penetrable infinite rough surfaces \cite{bl14is}, two-
and three-dimensional diffraction gratings \cite{bao14, bl14ip, Cheng13,
jl17ip}, bounded obstacles \cite{lw15ipi}, and interior cavities \cite{lw15cp}.
The two- and three-dimensional inverse elastic surface scattering problems have
been investigated by using near-field data in \cite{lwz-ip15, lwz16aa,
lwz-jcp16}. However, there exits some difficulties of near-field optical
imaging in practice, for example, it requires a sophisticated control of the
probe when scanning samples to measure the near-field data. Recently, a
rigorous mathematical model and an efficient numerical method are proposed in 
\cite{blw-aml16} to over the aforementioned obstacle in near-field imaging. The
novel idea is to put a rectangular slab of larger index of refraction above the
surfaces and allow more propagating wave modes to be able to propagate to the
far-field regime. This work is devoted to the inverse elastic surface
scattering problem with far-field data. We point out that this is a nontrivial
extension of the method from solving the inverse acoustic surface scattering
problem to solving the inverse elastic surface scattering problem, because the
latter involves the more complicated elastic wave equation due to the
coexistence of compressional and shear waves propagating at different speeds.

In this paper, we develop a rigorous mathematical model and an efficient
numerical method for the inverse elastic surface scattering with far-field
data. The scattering surface is assumed to be a small and smooth perturbation of
an elastically rigid plane. A rectangular slab of homogeneous and isotropic
elastic medium is placed above the scattering surface. The slab has a larger
mass density than that of the free space, and has a wavelength comparable
thickness. The measurement can be took on the top face of the slab, which is in
the far-field regime. The method makes use of the Helmholtz decomposition to
consider two coupled Helmholtz equations instead of the elastic wave equation.
It consists of two steps. The first step is to do the far-to-near (FtN) field
data conversion, which requires to solve a Cauchy problem of the Helmholtz
equation in the slab. Using the Fourier analysis, we compute the analytic
solution and find a formula connecting the wave fields on the top and bottom
faces of the slab: a larger mass density of the slab allows more propagating
wave modes to be converted stably from the far-field regime to the near-field
regime. The second step is to solve an inverse surface scattering problem in the
near-field zone by using the data obtained from the first step. Combining the
Fourier analysis, we use the transformed field expansions to find an analytic
solution for the direct problem. We refer to \cite{br93, mn11joasa, nr04,
nr04josaa, ls12} for the transformed field expansion and related boundary
perturbation methods for solving direct surface scattering problems. Using the
closed form of the analytic solution, we deduce expressions for the leading and
linear terms of the power series solution. Dropping all higher order terms, we
linearize the inverse problem and obtain explicit reconstruction formulas for
the surface function. Moreover, a nonlinear correction scheme is also developed
to improve the reconstruction. The method requires only a single illumination
and is implemented efficiently by the fast Fourier transform (FFT). Numerical
examples show it is effective and robust to reconstruct the scattering surfaces
with subwavelength resolution.

The remaining part of the paper is organized as follows.  The mathematical
model problem is formulated in Section \ref{se:model}. Sections \ref{se:hd} and
\ref{se:tbc} introduce the Helmholtz decomposition and the transparent
boundary condition, respectively. In Section \ref{sec:sd}, we show how to
convert the measured elastic wave data into the scattering data of the scalar
potentials introduced from the Helmholtz decomposition. In Section \ref{sec:rp},
a reduced problem is modeled in the slab and the analytic solution is
obtained to accomplish the FtN field data conversion. In Section \ref{sec:tfe},
the transformed field expansion and corresponding recursive boundary value
problems are presented. We give the reconstruction formulas for the inverse
problem in Section \ref{sec:ip}. Numerical experiments are presented in Section
\ref{sec:ne} to demonstrate the effectiveness of the proposed method. Finally,
we conclude some general remarks and directions for future research in Section
\ref{sec:con}.

\section{Model problem}\label{se:model}

Let us first introduce the problem geometry, which is shown in Figure
\ref{pg}. Consider an elastically rigid surface $\Gamma_f=\{\boldsymbol
x=(x, y)\in\mathbb R^2: y=f(x),\,0<x<\Lambda\}$, where $f$ is a periodic
Lipschitz continuous function with period $\Lambda$. The scattering surface
function $f$ is assumed to have the form
\begin{equation}\label{f}
 f(x)=\varepsilon g(x),
\end{equation}
where $\varepsilon>0$ is a sufficiently small constant and is called the surface
deformation parameter, $g$ is the surface profile function which is also
periodic with the period $\Lambda$. Hence the surface $\Gamma_f$ is a
small perturbation of the planar surface $\Gamma_0=\{\boldsymbol x\in\mathbb
R^2: y=0,\,0<x<\Lambda\}.$
Let a rectangular slab of homogeneous and isotropic elastic medium be placed
above the scattering surface. The bottom face of the slab is $
\Gamma_b=\{\boldsymbol x\in\mathbb R^2: y=b,\,0<x<\Lambda\},$
where $b>\max_{x\in (0, \Lambda)}f(x)$ is a constant and stands for the
separation distance between the scattering surface and the slab. The top face of
the slab is $\Gamma_a=\{\boldsymbol x\in\mathbb R^2: y=a,\,0<x<\Lambda\},$
where $a>b$ is a positive constant and stands for the measurement
distance. Denote by $\Omega$ the bounded domain between $\Gamma_f$ and
$\Gamma_b$, i.e., $\Omega=\{\boldsymbol x\in\mathbb R^2: f<y<b,\,0<x<\Lambda\}.$
Let $R$ be the domain of the slab, i.e., $R=\{\boldsymbol x\in\mathbb R^2:
b<y<a,\,0<x<\Lambda\}.$ Finally, denote by $U$ the open domain above $\Gamma_a$,
i.e., $U=\{\boldsymbol x\in\mathbb R^2: y>a,\,0<x<\Lambda\}.$

\begin{figure}
\centering
\includegraphics[width=0.35\textwidth]{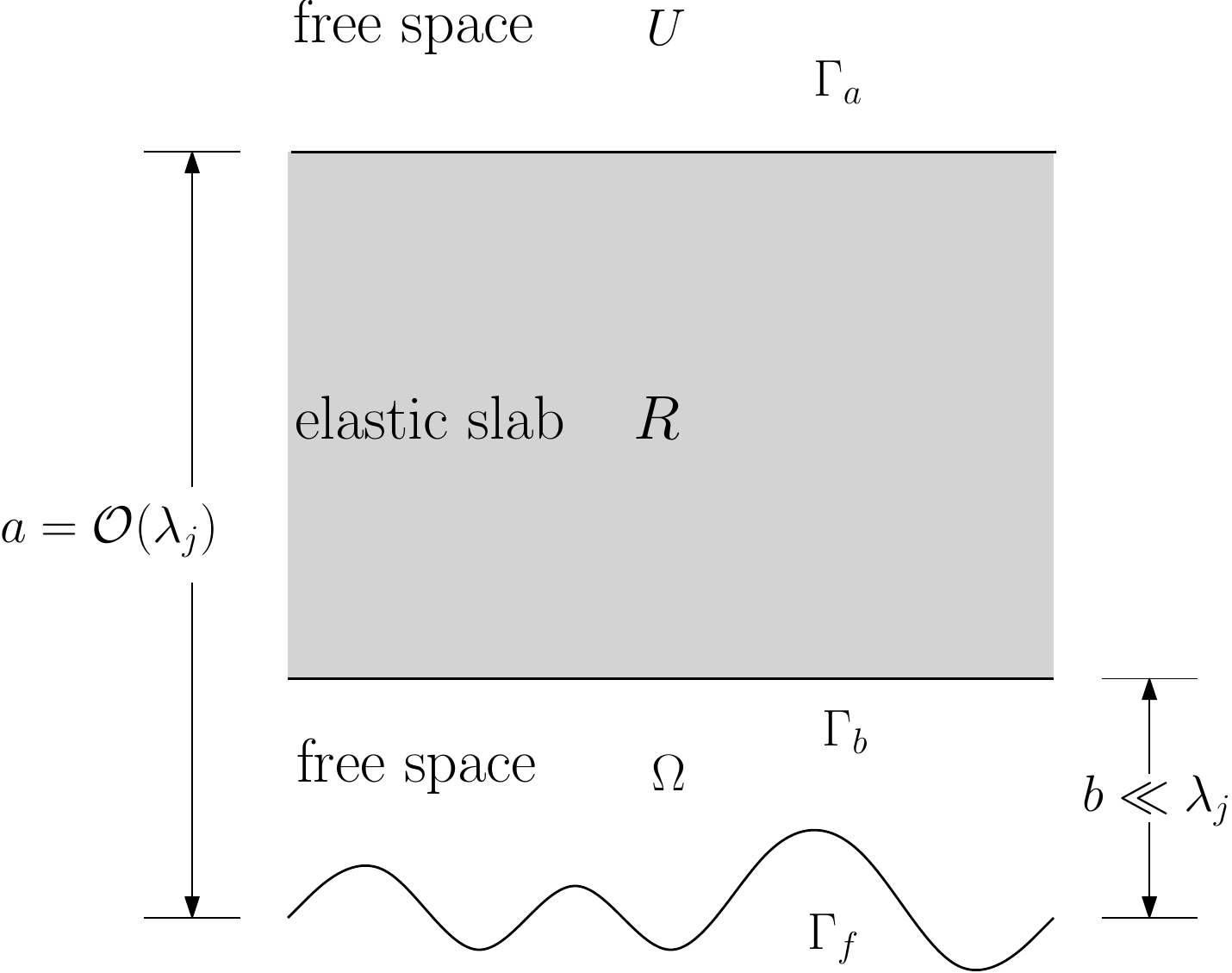}
\caption{The problem geometry. }
\label{pg}
\end{figure}

In this paper, we assume for simplicity that the Lam\'{e} parameters $\mu,
\lambda$ are constants satisfying $\mu>0, \lambda+\mu>0$; the mass density
$\rho$ is a piecewise constant, i.e.,
\[
 \rho(\boldsymbol x)=\begin{cases}
       \rho_0,\quad\boldsymbol x\in\Omega\cup U,\\
       \rho_1,\quad\boldsymbol x\in R,
      \end{cases}
\]
where $\rho_0$ and $\rho_1$ are the density of the free space and the elastic
slab, respectively, and they satisfy $\rho_1>\rho_0>0$. Define
\[
 \kappa_1=\omega \left(\frac{\rho_0}{\lambda+2\mu}\right)^{1/2},\quad
\kappa_2=\omega\left(\frac{\rho_0}{\mu}\right)^{1/2},
\]
which are known as the compressional wavenumber and the shear wavenumber in the
free space, respectively. We comment that the method also works for the case
where $\mu, \lambda$ take different values in the free space and the elastic
slab. Let $\lambda_j=2\pi/\kappa_j, j=1, 2$ be the
corresponding wavelength of the compressional and shear waves.

Let $\boldsymbol u^{\rm inc}$ be a time-harmonic plane wave which is incident
on the slab from above. The incident plane wave can be taken as either the
compressional wave $\boldsymbol u^{\rm inc}(\boldsymbol x)=\boldsymbol d e^{{\rm
i}\kappa_1\boldsymbol x\cdot\boldsymbol d}$ or the shear wave $\boldsymbol
u^{\rm inc}=\boldsymbol d^\perp e^{{\rm i}\kappa_2 \boldsymbol
x\cdot\boldsymbol d}$, where $\boldsymbol d=(\sin\theta, -\cos\theta)^\top$ is
the unit incident direction vector, $\theta\in (-\pi/2, \pi/2)$ is
the incident angle, and $\boldsymbol d^\perp=(\cos\theta, \sin\theta)^\top$ is
an orthonormal vector to $\boldsymbol d$. In this work, we use the compressional
incident plane wave as an example to present the results, which are similar and
can be obtained with obvious modifications for the shear incident plane
wave. Practically, the simplest configuration is the normal
incidence for experiments, i.e., $\theta=0$. Hence we focus on the
normal incidence since our method requires only a single illumination. Under the
normal incidence, the incident field reduces to
\begin{equation}\label{if}
 \boldsymbol u^{\rm inc}(\boldsymbol x)=(0, -1)^\top e^{-{\rm i}\kappa_1 y}.
\end{equation}
It can be verified that the incident field $\boldsymbol u^{\rm
inc}$ satisfies the elastic wave equation:
\begin{equation}\label{ui}
 \mu\Delta\boldsymbol u^{\rm inc}+(\lambda+\mu)\nabla\nabla\cdot\boldsymbol
u^{\rm inc}+\omega^2\rho_0\boldsymbol u^{\rm inc}=0\quad\text{in} ~ U.
\end{equation}

A transmission problem can be formulated due to the interaction between the
elastic wave and the interfaces $\Gamma_a$ and $\Gamma_b$. Let $\boldsymbol u,
\boldsymbol v, \boldsymbol w$ be the displacements of the total field
in the domains $U, R, \Omega$, respectively. They satisfy the elastic wave
equations:
\begin{subequations}\label{tf}
\begin{align}
\label{tfu} \mu\Delta\boldsymbol u+(\lambda+\mu)\nabla\nabla\cdot\boldsymbol
u+\omega^2\rho_0\boldsymbol u=0 &\quad\text{in} ~ U,\\
\label{tfv} \mu\Delta\boldsymbol v+(\lambda+\mu)\nabla\nabla\cdot\boldsymbol
v+\omega^2\rho_1\boldsymbol v=0 &\quad\text{in} ~ R,\\
\label{tfw}\mu\Delta\boldsymbol w+(\lambda+\mu)\nabla\nabla\cdot\boldsymbol
w+\omega^2\rho_0\boldsymbol w=0 &\quad\text{in} ~ \Omega.
\end{align}
\end{subequations}
In addition, the total fields are connected by the continuity conditions:
\begin{subequations}
\begin{align}
\label{cca}\boldsymbol u=\boldsymbol v, &\quad \mu\partial_y \boldsymbol
u+(\lambda+\mu)(0, 1)^\top\nabla\cdot\boldsymbol u=\mu\partial_y\boldsymbol
v+(\lambda+\mu)(0, 1)^\top\nabla\cdot\boldsymbol v \quad\text{on} ~
\Gamma_a,\\
\label{ccb}\boldsymbol v=\boldsymbol w, &\quad \mu\partial_y\boldsymbol
v+(\lambda+\mu)(0, 1)^\top \nabla\cdot\boldsymbol v=\mu\partial_y\boldsymbol
w+(\lambda+\mu)(0, 1)^\top\nabla\cdot\boldsymbol w \quad\text{on} ~
\Gamma_b.
\end{align}
\end{subequations}
Since $\Gamma_f$ is elastically rigid, we have the homogeneous Dirichlet
boundary condition:
\begin{equation}\label{dbc}
 \boldsymbol w=0\quad\text{on} ~\Gamma_f.
\end{equation}
In the open domain $U$, the total field $\boldsymbol u$ consists of the incident
field $\boldsymbol u^{\rm inc}$ and the diffracted field $\boldsymbol u^{\rm
d}$:
\begin{equation}\label{ut}
 \boldsymbol u=\boldsymbol u^{\rm inc}+\boldsymbol u^{\rm d},
\end{equation}
where $\boldsymbol u^{\rm d}$ is required to satisfy the bounded outgoing wave
condition.

Throughout, we assume that the measurement distance $a=\mathcal{O}(\lambda_j)$
and the separation distance $b\ll\lambda_j$, i.e., $a$ is comparable with the
wavelength and $\Gamma_a$ is put in the far-field region; $b$ is much smaller
than the wavelength and $\Gamma_b$ is put in the near-field region. Now we are
ready to formulate the inverse problem: Given the incident field $\boldsymbol
u^{\rm inc}$, the inverse problem is to determine the scattering surface $f$
from the far-field measurement of the total field $\boldsymbol u$ on $\Gamma_a$.

\section{The Helmholtz decomposition}\label{se:hd}

In this section, we introduce the Helmholtz decomposition for the total fields
by using scalar potential functions, and deduce the continuity conditions for
these scalar fields. Let $\boldsymbol u=(u_1, u_2)^\top$ and $u$ be a vector and
a scalar function, respectively. Introduce the scalar and vector curl
operators:
\[
 {\rm curl}\boldsymbol u=\partial_x u_2-\partial_y u_1, \quad {\bf
curl}u=(\partial_y u, -\partial_x u)^\top.
\]

For any solution $\boldsymbol u=(u_1, u_2)^\top$ of  \eqref{tfu}, the Helmholtz
decomposition reads
\begin{equation}\label{hdu}
 \boldsymbol u=\nabla\phi_1 +{\bf curl}\phi_2,
\end{equation}
where $\phi_j, j=1, 2$ are two scalar potential functions. Explicitly, we have
\begin{equation}\label{u}
 u_1=\partial_x \phi_1+\partial_y \phi_2, \quad u_2=\partial_y\phi_1
-\partial_x \phi_2.
\end{equation}
Substituting \eqref{hdu} into \eqref{tfu}  yields
\[
 \nabla\left((\lambda+2\mu)\Delta\phi_1+\omega^2\rho_0\phi_1
\right)+{\bf curl}\left(\mu\Delta\phi_2+\omega^2\rho_0\phi_2
\right)=0,
\]
which is fulfilled if $\phi_j$ satisfies
\begin{equation}\label{heu}
 \Delta\phi_j+\kappa^2_j\phi_j=0\quad\text{in} ~ U.
\end{equation}
Combining \eqref{heu} and \eqref{hdu}, we obtain
\[
 \phi_1=-\frac{1}{\kappa_1^2}\nabla\cdot\boldsymbol
u,\quad \phi_2=\frac{1}{\kappa_2^2} {\rm curl}\boldsymbol u,
\]
which give
\begin{equation}\label{phi}
 \partial_x u_1+\partial_y u_2=-\kappa_1^2\phi_1,\quad \partial_x
u_2-\partial_y u_1=\kappa_2^2\phi_2.
\end{equation}

For any solution $\boldsymbol v=(v_1, v_2)^\top$ of \eqref{tfv}, we introduce
the Helmholtz decomposition by using scalar functions $\psi_j$:
\begin{equation}\label{hdv}
 \boldsymbol v=\nabla\psi_1 +{\bf curl}\psi_2,
\end{equation}
which gives explicitly that
\begin{equation}\label{v}
 v_1=\partial_x\psi_1+\partial_y \psi_2,\quad v_2=\partial_y\psi_1
-\partial_x\psi_2.
\end{equation}
Plugging \eqref{hdv} into \eqref{tfv}, we may have
\begin{equation}\label{hev}
 \Delta\psi_j+\eta_j^2\psi_j=0\quad\text{in} ~ R,
\end{equation}
where $\eta_1$ and $\eta_2$ are the compressional and shear wavenumbers in the
elastic slab, respectively, and are given by
\begin{equation}\label{eq:eta}
	\eta_1=\omega\left(\frac{\rho_1}{\lambda+2\mu}\right)^{1/2},\quad
\eta_2=\omega\left(\frac{\rho_1}{\mu}\right)^{1/2}.
\end{equation}
Combing \eqref{hev} and \eqref{hdv}, we get
\[
 \psi_1=-\frac{1}{\eta_1^2}\nabla\cdot\boldsymbol v,\quad
\psi_2=\frac{1}{\eta_2^2}{\rm curl}\boldsymbol v,
\]
which give
\begin{equation}\label{psi}
 \partial_x v_1+\partial_y v_2=-\eta_1^2\psi_1,\quad \partial_x
v_2-\partial_y v_1=\eta_2^2\psi_2.
\end{equation}

Since $\Gamma_a$ is a horizontal line, it is easy to verify from the continuity
condition \eqref{cca} that
\begin{equation}\label{ccuv}
 u_j=v_j, \quad \partial_y u_j = \partial_y v_j.
\end{equation}
Using \eqref{phi}, \eqref{psi}--\eqref{ccuv}, we deduce the first
continuity condition for the scalar potentials on $\Gamma_a$:
\begin{equation}\label{cca1}
 \kappa_j^2\phi_j=\eta_j^2\psi_j.
\end{equation}
It follows from \eqref{u}, \eqref{v}, and \eqref{ccuv} that we deduce the
second continuity condition for the scalar potentials on $\Gamma_a$:
\begin{equation}\label{cca2}
 \partial_y\phi_1-\partial_x\phi_2=\partial_y\psi_1-\partial_x\psi_2, \quad
\partial_y\phi_2+\partial_x\phi_1=\partial_y\psi_2+\partial_x\psi_1.
\end{equation}

Similarly, for any solution $\boldsymbol w=(w_1, w_2)^\top$ of \eqref{tfw}, the
Helmholtz decomposition is
\begin{equation}\label{hdw}
 \boldsymbol w=\nabla\varphi_1+{\bf curl}\varphi_2.
\end{equation}
Substituting \eqref{hdw} into \eqref{tfw}, we may get
\begin{equation*}%\label{hew}
 \Delta\varphi_j+\kappa_j^2\varphi_j=0\quad\text{in} ~ \Omega.
\end{equation*}
Noting \eqref{ccb}, we may repeat the same steps and obtain the continuity
conditions on $\Gamma_b$:
\begin{equation}\label{ccb1}
 \eta_j^2\psi_j=\kappa_j^2\varphi_j
\end{equation}
and
\begin{equation}\label{ccb2}
 \partial_y\psi_1-\partial_x\psi_2=
\partial_y\varphi_1-\partial_x\varphi_2,\quad
\partial_y\psi_2+\partial_x\psi_1=\partial_y\varphi_2+\partial_x\varphi_1.
\end{equation}
Finally, it follows from the boundary condition \eqref{dbc} and the Helmholtz
decomposition \eqref{hdw} that
\begin{equation}\label{hdbc}
 \partial_x\varphi_1+\partial_y\varphi_2=0,\quad
\partial_y\varphi_1-\partial_x\varphi_2=0\quad\text{on} ~ \Gamma_f.
\end{equation}

\section{Transparent boundary condition}\label{se:tbc}

It follows from \eqref{ui}, \eqref{tfu}, and \eqref{ut} that the diffracted
field $\boldsymbol u^{\rm d}$ also satisfies the elastic wave equation:
\begin{equation}\label{ud}
  \mu\Delta\boldsymbol u^{\rm d}+(\lambda+\mu)\nabla\nabla\cdot\boldsymbol
u^{\rm d}+\omega^2\rho_0\boldsymbol u^{\rm d}=0 \quad\text{in} ~ U.
\end{equation}
Introduce the Helmholtz decomposition for the diffracted field $\boldsymbol
u^{\rm d}$:
\begin{equation}\label{hdud}
 \boldsymbol u^{\rm d}=\nabla\phi^{\rm d}_1+{\bf curl}\phi^{\rm d}_2,
\end{equation}
Substituting \eqref{hdud} into \eqref{ud} may yield
\begin{equation}\label{he}
 \Delta\phi^{\rm d}_j+\kappa^2_j\phi^{\rm d}_j=0\quad\text{in} ~ U.
\end{equation}

It follows from the uniqueness of the solution for the direct problem that
$\phi^{\rm d}_j$ is a periodic function with period $\Lambda$ and admits the
Fourier series expansion:
\begin{equation}\label{fse}
 \phi^{\rm d}_j(x, y)=\sum_{n\in\mathbb Z}\phi^{\rm d}_{jn}(y)e^{{\rm i}\alpha_n
x},
\end{equation}
where $\alpha_n=2n\pi/\Lambda$. Plugging \eqref{fse} into \eqref{he} yields
\begin{equation}\label{ode}
\partial^2_{yy} \phi^{\rm d}_{jn}(y)+\beta_{jn}^2
\phi^{\rm d}_{jn}(y)=0, \quad y>a,
\end{equation}
where
\[
 \beta_{jn}=\begin{cases}
                (\kappa_j^2-\alpha_n^2)^{1/2},&\quad |\alpha_n|<\kappa_j,\\
                {\rm i}(\alpha_n^2-\kappa_j^2)^{1/2},&\quad |\alpha_n|>\kappa_j.
               \end{cases}
\]
Here we assume that $\beta_{jn}\neq 0$ to exclude possible resonance.

Using the bounded outgoing wave condition, we may solve \eqref{ode} analytically
and obtain the solution of \eqref{he} explicitly:
\begin{equation}\label{re}
 \phi^{\rm d}_j(x, y)=\sum_{n\in\mathbb Z}\phi^{\rm d}_{jn}(a) e^{{\rm
i}(\alpha_n x+\beta_{jn}(y-a))},
\end{equation}
which is called the Rayleigh expansion for the scalar potential function
$\phi^{\rm d}_j$. Taking the normal derivative of \eqref{re} on $\Gamma_a$ gives
\begin{equation}\label{nd}
\partial_y\phi^{\rm d}_j(x, a)=\sum_{n\in\mathbb Z}{\rm
i}\beta_{jn}\phi^{\rm d}_{jn}(a)e^{{\rm i}\alpha_n x}.
\end{equation}

For a given periodic function $u(x)$ with period $\Lambda$, it has the Fourier
series expansion:
\[
 u(x)=\sum_{n\in\mathbb Z}u_n e^{{\rm i}\alpha_n x},\quad
u_n=\frac{1}{\Lambda}\int_0^\Lambda u(x)e^{-{\rm i}\alpha_n x}{\rm d}x.
\]
We define the boundary operator:
\[
 (\mathscr T_j u)(x)=\sum_{n\in\mathbb Z}{\rm i}\beta_{jn}u_n e^{{\rm
i}\alpha_n
x}.
\]
It is easy to verify from \eqref{nd} that
\begin{equation}\label{tbcj}
 \partial_y \phi^{\rm d}_j=\mathscr T_j\phi^{\rm d}_j\quad\text{on} ~ \Gamma_a.
\end{equation}

Recalling the incident field \eqref{if}, we may also consider the Helmholtz
decomposition for the incident field:
\begin{equation}\label{hdui}
 \boldsymbol u^{\rm inc}=\nabla\phi_1^{\rm inc}+{\bf curl}\phi_2^{\rm inc},
\end{equation}
which gives
\[
\phi_1^{\rm inc}=-\frac{1}{\kappa_1^2}\nabla\cdot\boldsymbol u^{\rm
inc}=-\frac{\rm i}{\kappa_1}e^{-{\rm i}\kappa_1 y}, \quad \phi_2^{\rm
inc}=\frac{1}{\kappa_2^2}{\rm curl}\boldsymbol u^{\rm inc}=0.
\]
A simple calculation yields
\[
 \partial_y\phi_1^{\rm inc}=-e^{-{\rm i}\kappa_1 a}, \quad
\mathscr T_1\phi_1^{\rm inc}=e^{-{\rm i}\kappa_1 a},
\]
which gives
\begin{equation}\label{tbcif}
 \partial_y\phi_1^{\rm inc}= \mathscr T_1\phi_1^{\rm inc}+g_1, \quad \partial_y
\phi_2^{\rm inc}=\mathscr T_2\phi_2^{\rm inc}+g_2.
\end{equation}
Here $g_1=-2e^{-{\rm i}\kappa_1 a}$ and $g_2=0$.

Letting $\phi_j=\phi_j^{\rm inc}+\phi^{\rm d}_j$ and recalling $\boldsymbol
u=\boldsymbol u^{\rm inc}+\boldsymbol u^{\rm d}$, we get \eqref{hdu} by adding
\eqref{hdui} and \eqref{hdud}. Moreover, we obtain the transparent boundary
condition for the total scalar potentials by combing \eqref{tbcj} and
\eqref{tbcif}:
\begin{equation}\label{tbctj}
 \partial_y \phi_j=\mathscr T_j\phi_j + g_j\quad\text{on} ~ \Gamma_a.
\end{equation}
It follows from \eqref{cca1}--\eqref{cca2} that
\begin{align}\label{tbca}
\partial_y\phi_1&=\partial_y\psi_1-\partial_x\psi_2+\partial_x\phi_2=\partial_y
\psi_1-\partial_x\psi_2+\left(\frac{\eta_2^2}{\kappa_2^2}
\right)\partial_x\psi_2\notag\\
&=\partial_y\psi_1+\left(\frac{\eta_2^2-\kappa_2^2}{
\kappa_2^2} \right)\partial_x\psi_2,\notag \\
\partial_y\phi_2&=\partial_y\psi_2+\partial_x\psi_1-\partial_x\phi_1=\partial_y
\psi_2+\partial_x\psi_1-\left(\frac{\eta_1^2}{\kappa_1^2}
\right)\partial_x\psi_1\notag\\
&=\partial_y\psi_2-\left(\frac{\eta_1^2-\kappa_1^2}{
\kappa_1^2}\right)\partial_x\psi_1.
\end{align}
Combining \eqref{tbctj}--\eqref{tbca} and \eqref{cca1} yields the boundary
condition for $\psi_j$ on $\Gamma_a$:
\begin{align}\label{tbcpj}
\partial_y\psi_1+\left(\frac{\eta_2^2-\kappa_2^2}{\kappa_2^2}
\right)\partial_x\psi_2&=\left(\frac{\eta_1^2}{\kappa_1^2}
\right)\mathscr T_1\psi_1+g_1 , \nonumber
\\
\partial_y\psi_2-\left(\frac{\eta_1^2-\kappa_1^2}{\kappa_1^2}
\right)\partial_x\psi_1&=\left(\frac
{\eta_2^2}{\kappa_2^2}\right)\mathscr T_2\psi_2+g_2.
\end{align}

Let $\boldsymbol u$ be a periodic function of $x$ with period $\Lambda$. It
admits the Fourier series expansion:
\[
 \boldsymbol u(x)=\sum_{n\in\mathbb Z}\boldsymbol u_n e^{{\rm
i}\alpha_n x},\quad \boldsymbol u_n=\frac{1}{\Lambda}\int_0^\Lambda
\boldsymbol u(x) e^{-{\rm i}\alpha_n x}{\rm d}x.
\]
Define the boundary operator on $\Gamma_a$:
\[
 (\mathscr T\boldsymbol u)(x)=\sum_{n\in\mathbb Z}{\rm i}
 \begin{bmatrix}
  \frac{\omega^2\beta_{1n}}{\alpha_n^2+\beta_{1n}\beta_{2n}} &
\mu\alpha_n -\frac{\omega^2\alpha_n^2}{\alpha_n^2+\beta_{1n}\beta_{2n}}\\[5pt]
\frac{\omega^2\alpha_n^2}{\alpha_n^2+\beta_{1n}\beta_{2n}}-\mu\alpha_n
&  \frac{\omega^2\beta_{2n}}{\alpha_n^2+\beta_{1n}\beta_{2n}}
 \end{bmatrix}\boldsymbol u_n e^{{\rm i}\alpha_n x}.
\]
It is shown in \cite{lwz-ip15} that $\alpha_n^2+\beta_{1n}\beta_{2n}\neq
0$ for $n\in\mathbb Z$ and the diffracted field $\boldsymbol u^{\rm d}$
satisfies the transparent boundary condition:
\[
 \mu\partial_y\boldsymbol u^{\rm d}+(\lambda+\mu)(0, 1)^\top
\nabla\cdot\boldsymbol u^{\rm d}=\mathscr T\boldsymbol u^{\rm d}\quad\text{on}~
\Gamma_a.
\]
A simple calculation yields that
\[
 \mu\partial_y \boldsymbol u^{\rm inc}+(\lambda+\mu)(0,
1)^\top\nabla\cdot\boldsymbol u^{\rm inc}={\rm i}\kappa_1(\lambda+2\mu)(0,
1)^\top e^{-{\rm i}\kappa_1 a}
\]
and
\[
 \mathscr T\boldsymbol u^{\rm inc}=-{\rm i}\kappa_1(\lambda+2\mu)(0, 1)^\top
e^{-{\rm i}\kappa_1 a}.
\]
Hence we obtain the boundary condition for the total displacement field
$\boldsymbol u$:
\[
 \mu\partial_y\boldsymbol u+(\lambda+\mu)(0,1)^\top \nabla\cdot\boldsymbol
u=\mathscr T\boldsymbol u+\boldsymbol h\quad\text{on} ~ \Gamma_a,
\]
where $\boldsymbol h=2{\rm i}\kappa_1(\lambda+2\mu)(0, 1)^\top
e^{-{\rm i}\kappa_1 a}$. Noting the continuity condition \eqref{cca}, we have
\[
 \mu\partial_y\boldsymbol v+(\lambda+\mu)(0,1)^\top \nabla\cdot\boldsymbol
v=\mathscr T\boldsymbol v+\boldsymbol h\quad\text{on} ~ \Gamma_a.
\]

\section{Scattering data}\label{sec:sd}

We assume that the total field $\boldsymbol u$ is measured
on $\Gamma_a$, i.e., $\boldsymbol u(x, a)=(u_1(x, a), u_2(x, a))^\top$
is available for $x\in (0, \Lambda)$. In this section, we show how to convert
$\boldsymbol u(x, a)$ into the scattering data of the scalar potentials
$\phi_j(x, a)$.

Evaluating \eqref{u} on $\Gamma_a$, we have
\begin{equation}\label{ua}
 \partial_x \phi_1(x, a) +\partial_y\phi_2(x, a)=u_1(x, a),\quad
\partial_y\phi_1(x, a)-\partial_x\phi_2(x, a)=u_2(x, a).
\end{equation}
Let $\phi_j(x, a)$ admit the Fourier series expansion
\begin{equation}\label{fc}
 \phi_j(x, a)=\sum_{n\in\mathbb Z}\phi_{jn} e^{{\rm i}\alpha_n x}.
\end{equation}
It suffices to find all the Fourier coefficients of $\phi_{jn}$ in order to
determine $\phi_j(x, a)$.

Taking the derivative of \eqref{fc} with respect to $x$ yields
\begin{equation}\label{dxfc}
 \partial_x\phi_j(x, a)=\sum_{n\in\mathbb Z}{\rm i}\alpha_n \phi_{jn}
e^{{\rm i}\alpha_n x}.
\end{equation}
It follows from the transparent boundary condition \eqref{tbctj} that
\begin{equation}\label{dyfc}
 \partial_y \phi_j(x, a)=\sum_{n\in\mathbb Z}{\rm
i}\beta_{jn}\phi_{jn}e^{{\rm i}\alpha_n x}+g_j.
\end{equation}
Substituting \eqref{dxfc} and \eqref{dyfc} into \eqref{ua}, we obtain a linear
system of equations for the Fourier coefficients $\phi_{jn}$:
\begin{equation}\label{ls}
{\rm i}
 \begin{bmatrix}
  \alpha_n & \beta_{2n}\\
  \beta_{1n} & -\alpha_n
 \end{bmatrix}
 \begin{bmatrix}
  \phi_{1n}\\
  \phi_{2n}
 \end{bmatrix}
 =\begin{bmatrix}
   p_{1n}\\
   p_{2n}
  \end{bmatrix},
\end{equation}
where $p_{1n}=u_{1n}-g_{2n}$, $p_{2n}=u_{2n}-g_{1n}$  and
$u_{jn}$ are the Fourier coefficients of $u_j$, i.e.,
\[
 u_{jn}=\frac{1}{\Lambda}\int_0^\Lambda u_j(x, a) e^{-{\rm i}\alpha_n x}{\rm
d}x
\]
and
\[
 g_{1n}=\begin{cases}
            -2e^{-{\rm i}\kappa_1 a}&\quad\text{for} ~ n=0,\\
            0 &\quad\text{for} ~ n\neq 0,
           \end{cases}
           \quad
          g_{2n}=0~\text{for}~ n\in\mathbb Z.
\]
Using Cramer's rule, we obtain the unique solution of \eqref{ls}:
\begin{equation}\label{phin}
 \phi_{1n}=-{\rm i}\left(\frac{\alpha_n
p_{1n}+\beta_{2n}p_{2n}}{\alpha_n^2+\beta_{1n}\beta_{2n}}\right),
\quad \phi_{2n}={\rm i}\left(\frac{\alpha_n
p_{2n}-\beta_{1n}p_{1n}}{\alpha_n^2+\beta_{1n} \beta_{2n}}\right).
\end{equation}

Hence, we may assume that $\phi_j(x, a), j=1, 2$ are measured data. From now
on, we shall only work on the potential functions.

\section{Reduced problem}\label{sec:rp}

Recall the continuity condition \eqref{cca1} and the boundary condition
\eqref{tbcpj}. Given the data $\phi_j$ on $\Gamma_a$, we consider the Cauchy
problem for $\psi_j$:
\begin{subequations}\label{cp}
\begin{align}
 \Delta\psi_j+\eta_j^2\psi_j=0 &\quad\text{in}~ R,\\
 \psi_j=\left(\frac{\kappa_j^2}{\eta_j^2}\right)\phi_j&\quad\text{on} ~
\Gamma_a, \\
\partial_y\psi_1+\left(\frac{\eta_2^2-\kappa_2^2}{\kappa_2^2}
\right)\partial_x\psi_2=\left(\frac{\eta_1^2}{\kappa_1^2}
\right)\mathscr T_1\psi_1+g_1&\quad\text{on}~\Gamma_a, \\
\partial_y\psi_2-\left(\frac{\eta_1^2-\kappa_1^2}{\kappa_1^2}
\right)\partial_x\psi_1=\left(\frac{\eta_2^2}{\kappa_2^2}
\right)\mathscr T_2\psi_2+g_2&\quad\text{on}~\Gamma_a.
\end{align}
 \end{subequations}
Since $\psi_j$ is a periodic function of $x$, it has the Fourier series
expansion
\begin{equation}\label{fsp}
 \psi_j(x, y)=\sum_{n\in\mathbb Z}\psi_{jn}(y)e^{{\rm i}\alpha_n x}.
\end{equation}
Substituting  \eqref{fsp} into \eqref{cp}, we obtain a final value problem for
the second order equation in the frequency domain:
\begin{subequations}\label{cpf}
 \begin{align}
\partial^2_{yy}\psi_{jn}(y)+\gamma_{jn}^2\psi_{jn}(y)=0, &\quad
b<y<a,  \\
\psi_{jn}(a)=\left(\frac{\kappa_j^2}{\eta_j^2}\right)\phi_{jn}, &\quad
y=a,  \\
\partial_y \psi_{1n}(a)+{\rm
i}\alpha_n\left(\frac{\eta_2^2-\kappa_2^2}{\kappa_2^2}
\right)\psi_{2n}(a)={\rm i}\beta_{1n}\left(\frac{\eta_1^2}{\kappa_1^2}
\right)\psi_{1n}(a)+g_{1n},&\quad y=a, \\
\partial_y\psi_{2n}(a)-{\rm i}\alpha_n\left(\frac{\eta_1^2-\kappa_1^2}{
\kappa_1^2} \right)\psi_{1n}(a)={\rm
i}\beta_{2n}\left(\frac{\eta_2^2}{\kappa_2^2}
\right)\psi_{2n}(a)+g_{2n},&\quad y=a,
 \end{align}
 \end{subequations}
where $\phi_{jn}$ is given in \eqref{phin} and
\[
 \gamma_{jn}=\begin{cases}
           (\eta_j^2-\alpha_n^2)^{1/2}, &\quad |\alpha_n| < \eta_j,\\
           {\rm i}(\alpha_n^2-\eta_j^2)^{1/2}, &\quad |\alpha_n|>\eta_j.
          \end{cases}
\]
Again we assume that $\gamma_{jn}\neq 0$ to exclude possible resonance.

Using the continuity condition \eqref{cca1} again, we may further reduce
\eqref{cpf} into the following final value problem:
\begin{subequations}\label{fvp}
 \begin{align}
\partial^2_{yy}\psi_{jn}(y)+\gamma_{jn}^2\psi_{jn}(y)=0, &\quad
b<y<a,\\
\psi_{jn}=\hat\phi_{jn}, &\quad y=a, \\
\partial_y\psi_{jn}-{\rm
i}\hat\beta_{jn}\psi_{jn}=\hat{g}_{jn},&\quad y=a,
 \end{align}
 \end{subequations}
where
\[
 \hat\phi_{jn}=\left(\frac{\kappa_j^2}{\eta_j^2}\right)\phi_{jn},\quad
\hat\beta_{jn}=\left(\frac{\eta_j^2}{\kappa_j^2}\right)\beta_{jn}
\]
and
\begin{align*}
 \hat{g}_{1n}&=g_{1n}-{\rm i}\alpha_n\left(\frac{\eta_2^2-\kappa_2^2}{
\eta_2^2}\right)\phi_{2n},\\
\hat{g}_{2n}&=g_{2n}+{\rm i}\alpha_n\left(\frac{\eta_1^2-\kappa_1^2}{
\eta_1^2} \right)\phi_{1n}.
\end{align*}
It follows from Lemma \eqref{A1} that the final value problem \eqref{fvp} has a
unique solution which is
\begin{align}\label{fvps}
 \psi_{jn}(y)=&(2\gamma_{jn}^{-1})\left((\gamma_{jn}+\hat{\beta}_{jn} )\hat
{\phi}_{jn}-{\rm i}\hat{g}_{jn} \right)e^{-{\rm i}\gamma_{jn}(a-y)}\notag\\
&+(2\gamma_{jn})^{-1}\left((\gamma_{jn}-\hat{\beta}_{jn})\hat
{\phi}_{jn}+{\rm i}\hat{g}_{jn} \right)e^{{\rm i}\gamma_{jn}(a-y)}.
\end{align}
Evaluating \eqref{fvps} at $y=b$ yields
\begin{align}\label{fvpb}
 \psi_{jn}(b)=&(2\gamma_{jn})^{-1})\left((\gamma_{jn}+\hat{\beta}_{jn})\hat
{\phi}_{jn}-{\rm i}\hat{g}_{jn} \right)e^{-{\rm i}\gamma_{jn}(a-b)}\notag\\
&+(2\gamma_{jn})^{-1}\left((\gamma_{jn}-\hat{\beta}_{jn})\hat
{\phi}_{jn}+{\rm i}\hat{g}_{jn} \right)e^{{\rm i}\gamma_{jn}(a-b)}.
\end{align}
where $\psi_{jn}(b)$ are the Fourier coefficients of $\psi_j(x, b)$. Taking
the partial derivative of \eqref{fvps} with respect to $y$ and
evaluating it at $y=b$, we obtain
\begin{align}\label{pfvpb}
 \partial_y \psi_{jn}(b)=&
\frac{\rm i}{2}\left((\gamma_{jn}+\hat{\beta}_{jn})\hat{\phi}_{jn}-{\rm
i}\hat{g}_{jn} \right)e^{-{\rm i}\gamma_{jn}(a-b)}\notag\\
&-\frac{\rm i}{2}\left((\gamma_{jn}-\hat{\beta}_{jn})\hat
{\phi}_{jn}+{\rm i}\hat{g}_{jn} \right)e^{{\rm
i}\gamma_{jn}(a-b)}.
\end{align}

We point out that \eqref{fvpb} gives the far-to-near (FtN) field data conversion
formula. We observe from \eqref{fvpb} that it is stable to convert the far-field
data for the propagating wave components where the Fourier modes satisfy
$|\alpha_n|<\eta_j$; it is exponentially unstable to convert the far-field for
the evanescent wave components where the Fourier modes satisfy
$|\alpha_n|>\eta_j$. Thus it is only reliable to make the near-field data by
converting the low frequency far-field data $\phi_{jn}$ with
$|\alpha_n|<\eta_j$. Noting $\rho_1>\rho_0$ in the elastic slab, we are allowed
to include more propagating wave modes to reconstruct the surface than the case
without the slab, which contributes to a better resolution.

It follows from the continuity condition \eqref{ccb1} that
\begin{equation}\label{vb}
 \varphi_{jn}(b)=\left(\frac{\eta_j^2}{\kappa_j^2}\right)\psi_{jn}(b).
\end{equation}
Using the continuity conditions \eqref{ccb1}--\eqref{ccb2} on $\Gamma_b$, we
obtain
\begin{subequations}
\begin{align}
 \partial_y\varphi_1&=\partial_y\psi_1-\partial_x\psi_2+\partial_x\varphi_2
 =\partial_y\psi_1-\partial_x\psi_2+\left(\frac{\eta_2^2}{\kappa_2^2}
\right)\partial_x\psi_2\notag\\
&=\partial_y\psi_1+\left(\frac{\eta_2^2-\kappa_2^2}{\kappa_2^2}
\right)\partial_x\psi_2,\notag\\
\partial_y\varphi_2&=\partial_y\psi_2+\partial_x\psi_1-\partial_x\varphi_1
=\partial_y\psi_2+\partial_x\psi_1-\left(\frac{\eta_1^2}{\kappa_1^2}
\right)\partial_x\psi_1\notag\\
&=\partial_y\psi_2-\left(\frac{\eta_1^2-\kappa_1^2}{\kappa_1^2}
\right)\partial_x\psi_1,\notag
\end{align}
\end{subequations}
which give in the frequency domain that
\begin{align}\label{dvb}
 \partial_y\varphi_{1n}(b)&=\partial_y\psi_{1n}(b)+{\rm i}\alpha_n
\left(\frac{\eta_2^2-\kappa_2^2}{\kappa_2^2}\right)\psi_{2n}(b), \nonumber \\
\partial_y\varphi_{2n}(b)&=\partial_y\psi_{2n}(b)-{\rm
i}\alpha_n\left(\frac{\eta_1^2-\kappa_1^2}{\kappa_1^2}\right)\psi_{1n}(b).
\end{align}
Combining \eqref{vb} and \eqref{dvb}, we get
\begin{equation}\label{htbc}
 (\partial_y-{\rm i}\beta_{jn})\varphi_{jn}=  \tau_{jn},
\end{equation}
where
\begin{align}\label{tau}
 \tau_{1n}&=\partial_y\psi_{1n}(b)-{\rm
i}\hat{\beta}_{1n}\psi_{1n}(b)+{\rm i}\alpha_n
\left(\frac{\eta_2^2-\kappa_2^2}{\kappa_2^2}\right)\psi_{2n}(b),\nonumber \\
\tau_{2n}&=\partial_y\psi_{2n}(b)-{\rm
i}\hat{\beta}_{2n}\psi_{2n}(b)-{\rm i}\alpha_n
\left(\frac{\eta_1^2-\kappa_1^2}{\kappa_1^2}\right)\psi_{1n}(b).
\end{align}
Here the Fourier coefficients $\psi_{jn}(b)$ and $\partial_y\psi_{jn}(b)$
are given in \eqref{fvpb} and \eqref{pfvpb}, respectively.

Using the boundary conditions \eqref{hdbc} and \eqref{htbc}, we may consider the
following reduced boundary value problem for the scalar potential $\varphi_j$ in
$\Omega$:
\begin{subequations}\label{bvp}
\begin{align}
\label{fhe} \Delta\varphi_j + \kappa_j^2\varphi_j=0 &\quad\text{in}~\Omega,\\
\label{fdbc} \partial_x\varphi_1+\partial_y\varphi_2=0, \quad
\partial_y\varphi_1-\partial_x\varphi_2=0 &\quad\text{on}~ \Gamma_f,\\
\label{ftbc}\partial_y\varphi_j=\mathscr T_j\varphi_j+\tau_j &\quad\text{on} ~
\Gamma_b,
\end{align}
\end{subequations}
where the Fourier coefficients of $\tau_j$ are given in \eqref{tau}. The
inverse problem is reformulated to determine the periodic scattering surface
function $f$ from the Fourier coefficients $\varphi_{jn}(b)$ for $n\in
M_j=\{n\in\mathbb Z: |\alpha_n|<\eta_j\}$.

\section{Transformed field expansion}\label{sec:tfe}

In this section, we introduce the transformed field expansion to derive an
analytic solution to the boundary value problem \eqref{bvp}.

\subsection{Change of variables}

Consider the change of variables:
\[
  \tilde{x} = x, \quad \tilde{y} = b \left( \frac{y-f}{b-f} \right),
\]
which maps $\Gamma_f$ to $\Gamma_0$ but keeps $\Gamma_b$ unchanged. Hence the
domain $\Omega$ is mapped into the rectangular domain
$D=\{(\tilde{x}, \tilde{y})\in\mathbb{R}^2: 0<\tilde{x}<\Lambda,
\,0<\tilde{y}<b\}$. It is easy to verify the differential rules:
\begin{align*}
\partial_x = & \partial_{\tilde{x}} - f' \left( \frac{b-\tilde{y}}{b-f} \right)
\partial_{\tilde{y}},\\
    \partial_y = &\left( \frac{b}{b-f} \right) \partial_{\tilde{y}},\\
    \partial^2_{xx} = & \partial^2_{\tilde{x}\tilde{x}} + (f')^2 \left(
\frac{b-\tilde{y}}{b-f} \right)^2 \partial^2_{\tilde{y}\tilde{y}} - 2 f' \left(
\frac{b-\tilde{y}}{b-f} \right) \partial^2_{\tilde{x}\tilde{y}} \\
    &- \left[ f'' \left( \frac{b-\tilde{y}}{b-f} \right) + 2 (f')^2
\frac{(b-\tilde{y})}{(b-f)^2} \right] \partial_{\tilde{y}}, \\
    \partial^2_{yy} = & \left( \frac{b}{b-f}
\right)^2\partial^2_{\tilde{y}\tilde{y}}.
\end{align*}

We introduce a function $\tilde{\varphi}_j(\tilde{x}, \tilde{y})$ in order to
reformulate the boundary value problem (\ref{bvp}) using the new variables.
Noting \eqref{fhe}, we have from the straightforward calculations that
$\tilde{\varphi}$, upon dropping the tilde for simplicity of notation,
satisfies
\begin{equation}\label{tdf}
 \left(c_1 \partial^2_{xx} + c_2 \partial^2_{yy} + c_3 \partial^2_{xy} + c_4
\partial_y + c_1 \kappa^2_j \right) \varphi_j = 0 \quad\text{in} ~ D,
\end{equation}
where
\begin{equation}\label{c}
  \begin{cases}
    c_1 = (b-f)^2,\\
    c_2 = \left[ f' (b-y) \right]^2 + b^2,\\
    c_3 = -2 f' (b-y) (b-f),\\
    c_4 = - (b-y) \left[ f''(b-f) + 2 (f')^2 \right].
  \end{cases}
\end{equation}
The boundary condition \eqref{fdbc} becomes
\begin{equation}\label{tdbc}
     \left[ \left( 1 - b^{-1} f \right) \partial_x - f' \partial_y \right]
\varphi_1+ \partial_y \varphi_2 = 0, \quad
    \partial_y \varphi_1 - \left[ \left( 1 - b^{-1} f \right) \partial_x - f'
\partial_y \right] \varphi_2 = 0.
\end{equation}
The boundary condition \eqref{ftbc} reduces to
\begin{equation}\label{ttbc}
\partial_y \varphi_j = \left( 1 - b^{-1} f \right) (\mathscr T_j\varphi_j +
\tau_j).
\end{equation}

\subsection{Power series expansion}

Noting the surface function \eqref{f}, we resort to the perturbation technique
and consider formal power series expansion of $\varphi_j$ in terms of
$\varepsilon$:
\begin{equation}\label{ps}
 \varphi_j(x,y; \varepsilon)=\sum_{k=0}^\infty \varphi_j^{(k)}(x, y)
\varepsilon^k.
\end{equation}
Substituting \eqref{f} into \eqref{c} and plugging \eqref{ps} into \eqref{tdf},
we may obtain the recurrence equations for $\varphi_j^{(k)}$ in $D$:
\begin{equation}\label{the}
\Delta\varphi_j^{(k)} + \kappa_j^2\varphi_j^{(k)} = u_j^{(k)},
\end{equation}
where
\begin{equation}\label{ujk}
  u_j^{(k)}=\mathscr{D}^{(1)}_j\varphi_j^{(k-1)} +
\mathscr{D}^{(2)}_j\varphi_j^{(k-2)}.
\end{equation}
Here the differential operators are 
\begin{align*}
 \mathscr{D}_j^{(1)} = & b^{-1} \left[ 2g \partial^2_{xx} + 2g'(b-y)
\partial^2_{xy} + g''(b-y) \partial_y + 2 \kappa_j^2 g \right],\\
    \mathscr{D}_j^{(2)} = & - b^{-2} \left\{ g^2 \partial^2_{xx} + (g')^2
(b-y)^2 \partial^2_{yy} + 2gg'(b-y) \partial^2_{xy} \right. \\
    &\left. - \left[ 2(g')^2 - gg'' \right](b-y) \partial_y + \kappa_j^2
g^2 \right\}.
\end{align*}
Substituting \eqref{f} and \eqref{ps} into \eqref{tdbc}, we obtain the
recurrence equations for the boundary conditions on $\Gamma_0$:
\[
 \partial_x \varphi_1^{(k)} + \partial_y \varphi_2^{(k)}= p^{(k)},\quad
 \partial_y \varphi_1^{(k)} - \partial_x \varphi_2^{(k)} = q^{(k)},
\]
where
\begin{equation}\label{pq}
  p^{(k)}= \left( b^{-1} g \partial_x + g' \partial_y \right)
\varphi_1^{(k-1)},\quad
  q^{(k)}=- \left( b^{-1} g \partial_x + g' \partial_y \right)
\varphi_2^{(k-1)}.
\end{equation}
Substituting \eqref{f} and \eqref{ps} into \eqref{ttbc}, we derive the
recurrence equations for the transparent boundary conditions on $\Gamma_b$:
\[
\left( \partial_y - \mathscr T_j \right) \varphi_j^{(k)} =r_j^{(k)},
\]
where
\begin{equation}\label{rs}
 r_j^{(0)}=\tau_j, \quad r_j^{(1)}=-b^{-1} g (\mathscr T_j\varphi_j^{(0)} +
\tau_j),\quad r_j^{(k)}=-b^{-1}g \mathscr T_j\varphi_j^{(k-1)}.
\end{equation}

In all of the above recurrence equations, it is understood that
$\varphi_j^{(k)}, u_j^{(k)}, p^{(k)}, q^{(k)}, r_j^{(k)}$ are zeros when $k<0$.
The boundary value problem \eqref{the}--\eqref{rs} for the current terms
$\varphi_j^{(k)}$ involve $u_j^{(k)}, p^{(k)}, q^{(k)}, r_j^{(k)}$, which
depend only on previous two terms of $\varphi_j^{(k-1)}, \varphi_j^{(k-2)}$.
Thus, the boundary value problem \eqref{the}--\eqref{rs} can be recursively
solved from $k=0$.

\subsection{Fourier series expansion}

Since $\varphi_j^{(k)}$ are periodic functions of $x$ with period $\Lambda$,
they have the Fourier series expansions
\begin{equation}\label{fs}
 \varphi_j^{(k)}(x, y)=\sum_{n\in\mathbb{Z}}\varphi_{jn}^{(k)}(y)e^{{\rm
i}\alpha_n x}.
\end{equation}
Substituting \eqref{fs} into the boundary value problem
\eqref{the}--\eqref{rs}, we obtain a coupled two-point boundary value problems:
 \begin{align}\label{p1k}
  \partial^2_{yy}\varphi_{1n}^{(k)}  + \beta_{1n}^2 \varphi_{1n}^{(k)}
=u_{1n}^{(k)}, \quad & 0 < y < b,  \nonumber \\
    \partial_y\varphi_{1n}^{(k)}  = q_n^{(k)} +  {\rm i} \alpha_n
\varphi_{2n}^{(k)},\quad & y = 0,  \\
    \partial_y \varphi_{1n}^{(k)} - {\rm i}\beta_{1n}  \varphi_{1n}^{(k)}
=r_{1n}^{(k)}, \quad & y = b \nonumber
 \end{align}
and
 \begin{align}\label{p2k}
  \partial^2_{yy}\varphi_{2n}^{(k)}  + \beta_{2n}^2 \varphi_{2n}^{(k)}
=u_{2n}^{(k)}, \quad & 0 < y < b, \nonumber \\
    \partial_y\varphi_{2n}^{(k)}  = p_n^{(k)} -  {\rm i} \alpha_n
\varphi_{1n}^{(k)}, \quad & y = 0,  \\
    \partial_y \varphi_{2n}^{(k)} - {\rm i}\beta_{2n}\varphi_{2n}^{(k)}
=r_{2n}^{(k)}, \quad & y = b,\nonumber
 \end{align}
where $u_{jn}^{(k)}, p_n^{(k)}, q_n^{(k)}, r_{jn}^{(k)}$ are the Fourier
coefficients of $u_j^{(k)}, p^{(k)}, q^{(k)}, r_j^{(k)}$, respectively.

It follows from Lemma \ref{A2} that the solutions of \eqref{p1k} and
\eqref{p2k} are
\begin{subequations}\label{rvpk}
\begin{align}
 \varphi_{1n}^{(k)}(y)=&K_1(y; \beta_{1n})(q_n^{(k)}+{\rm
i}\alpha_n\varphi_{2n}^{(k)}(0))\notag\\
&-K_2(y; \beta_{1n})r_{1n}^{(k)}+\int_0^b K_3(y, z;
\beta_{1n})u_{1n}^{(k)}(z){\rm d}z,\label{rvpka} \\
\varphi_{2n}^{(k)}(y)=&K_1(y; \beta_{2n})(p_n^{(k)}-{\rm
i}\alpha_n\varphi_{1n}^{(k)}(0))\notag\\
&-K_2(y; \beta_{2n})r_{2n}^{(k)}+\int_0^b K_3(y, z;
\beta_{2n})u_{2n}^{(k)}(z){\rm d}z, \label{rvpkb}
\end{align}
\end{subequations}
where $\varphi_{jn}^{(k)}(0)$ are to be determined. Evaluating
$\varphi_{jn}^{(k)}(y)$ at $y=0$ in the above equations and recalling $K_j$ in
Lemma \ref{A2}, we obtain
\begin{align*}
 {\rm i}\beta_{1n}\varphi_{1n}^{(k)}(0)=(q_n^{(k)}+{\rm
i}\alpha_n\varphi_{2n}^{(k)}(0))-e^{{\rm i}\beta_{1n}b}r_{1n}^{(k)} +
\int_0^b e^{{\rm i}\beta_{1n}z} u_{1n}^{(k)}(z){\rm d}z,\\
{\rm i}\beta_{2n}\varphi_{2n}^{(k)}(0)=(p_n^{(k)}-{\rm
i}\alpha_n\varphi_{1n}^{(k)}(0))-e^{{\rm i}\beta_{2n}b}r_{2n}^{(k)} +
\int_0^b e^{{\rm i}\beta_{2n}z} u_{2n}^{(k)}(z){\rm d}z,
\end{align*}
which yields a system of algebraic equations for $\varphi_{jn}^{(k)}(0)$:
\begin{equation}\label{lsfc}
{\rm i} \begin{bmatrix}
  \beta_{1n} & -\alpha_n\\[5pt]
  \alpha_n & \beta_{2n}
 \end{bmatrix}
 \begin{bmatrix}
  \varphi_{1n}^{(k)}(0)\\[5pt]
  \varphi_{2n}^{(k)}(0)
 \end{bmatrix}
 =\begin{bmatrix}
   v^{(k)}_{1n}\\[5pt]
   v^{(k)}_{2n}
   \end{bmatrix},
\end{equation}
where
\begin{align*}
 v^{(k)}_{1n}&=q_n^{(k)} - e^{{\rm
i}\beta_{1n} b} r_{1n}^{(k)} + \int_0^b e^{{\rm i}\beta_{1n} z}
u_{1n}^{(k)}(z){\rm d}z,\\
v^{(k)}_{2n}&=p_n^{(k)} - e^{{\rm i}\beta_{2n} b} r_{2n}^{(k)} +
\int_0^b e^{{\rm i}\beta_{2n} z} u_{2n}^{(k)}(z){\rm d}z.
\end{align*}
It follows from Cramer's rule again that the linear system has a unique solution
which is given by
\[
\varphi_{1n}^{(k)}(0)=-{\rm i}\left(\frac{\beta_{2n} v^{(k)}_{1n}+\alpha_n
v^{(k)}_{2n}}{\alpha_n^2 + \beta_{1n}\beta_{2n}}
\right),\quad \varphi_{2n}^{(k)}(0)=-{\rm i}\left(\frac{\beta_{1n}
v^{(k)}_{2n}-\alpha_n v^{(k)}_{1n}}{\alpha_n^2 + \beta_{1n}\beta_{2n}}
\right).
\]
Once $\varphi_{jn}^{(k)}(0)$ are determined, $\varphi_{jn}^{(k)}(y)$ can be
computed from \eqref{rvpka} and   \eqref{rvpkb} explicitly for all $k$ and $n$.

\subsection{Leading terms}

For $k=0$, it follows from \eqref{ujk}, \eqref{pq}, and \eqref{rs} that we
obtain
\[
 u_j^{(0)}=p^{(0)}=q^{(0)}=0, \quad r_j^{(0)}=\tau_j.
\]
Their Fourier coefficients are
\begin{equation}\label{fc0}
u_{jn}^{(0)}=p_n^{(0)}=q_n^{(0)}=0, \quad r_{jn}^{(0)}=\tau_{jn}.
\end{equation}
Substituting \eqref{fc0} into \eqref{lsfc} yields
\[
 v_{jn}^{(0)}=-e^{{\rm i}\beta_{jn}b}\tau_{jn}
\]
and
\begin{align}\label{vp0}
 \varphi_{1n}^{(0)}(0)=\left(\frac{{\rm i}\beta_{2n}e^{{\rm
i}\beta_{1n}b}}{\alpha_n^2+\beta_{1n}\beta_{2n}}\right)\tau_{1n}
+\left(\frac{{\rm i}\alpha_n e^{{\rm
i}\beta_{2n}b}}{\alpha_n^2+\beta_{1n}\beta_{2n}}\right)
\tau_{2n}, \nonumber \\
\varphi_{2n}^{(0)}(0)=\left(\frac{{\rm i}\beta_{1n}e^{{\rm
i}\beta_{2n}b}}{\alpha_n^2+\beta_{1n}\beta_{2n}}\right)\tau_{2n}
-\left(\frac{{\rm i}\alpha_n e^{{\rm
i}\beta_{1n}b}}{\alpha_n^2+\beta_{1n}\beta_{2n}}\right)
\tau_{1n}. 
\end{align}
Plugging \eqref{vp0} into \eqref{rvpk}, we get
\begin{subequations}
 \begin{align}
  \varphi_{1n}^{(0)}(y)&={\rm i}\alpha_n K_1(y,
\beta_{1n})\varphi_{2n}^{(0)}(0)-K_2(y, \beta_{1n})\tau_{1n}\notag\\
&=M_{1 1}^{(n)}(y)\tau_{1n}+M_{1 2}^{(n)}(y)\tau_{2n}, \label{rvp0a} \\
\varphi_{2n}^{(0)}(y)&=-{\rm i}\alpha_n K_1(y;
\beta_{2n})\varphi_{1n}^{(0)}(0)-K_2(y; \beta_{2n})\tau_{2n}\notag\\
&=M_{21}^{(n)}(y)\tau_{1n}+M_{2 2}^{(n)}(y)\tau_{2n}, \label{rvp0b}
 \end{align}
\end{subequations}
where
\begin{align*}
 M_{11}^{(n)}(y)&=-\left(\frac{{\rm i}\alpha^2_n
e^{{\rm i}\beta_{1n} b}}{\beta_{1n}(\alpha_n^2+\beta_{1n}\beta_{2n})}
\right)e^{{\rm i}\beta_{1n}y}+\frac{{\rm
i}e^{{\rm i}\beta_{1n}b}}{2\beta_{1n}}(e^{{\rm i}\beta_{1n}
y}+e^{-{\rm i}\beta_{1n} y}),\\
M_{12}^{(n)}(y)&=\left(\frac{{\rm i}\alpha_n e^{{\rm
i}\beta_{2n} b}}{\alpha_n^2+\beta_{1n}\beta_{2n}}
\right)e^{{\rm i}\beta_{1n}y},\\
M_{21}^{(n)}(y)&=-\left(\frac{{\rm i}\alpha_n e^{{\rm
i}\beta_{1n} b}}{\alpha_n^2+\beta_{1n}\beta_{2n}}
\right)e^{{\rm i}\beta_{2n}y},\\
M_{22}^{(n)}(y)&=-\left(\frac{{\rm i}\alpha^2_n
e^{{\rm i}\beta_{2n} b}}{\beta_{2n}(\alpha_n^2+\beta_{1n}\beta_{2n})}
\right)e^{{\rm i}\beta_{2n}y}+\frac{{\rm
i}e^{{\rm i}\beta_{2n}b}}{2\beta_{2n}}(e^{{\rm i}\beta_{2n}
y}+e^{-{\rm i}\beta_{2n}y}).
\end{align*}

\subsection{Linear terms}

For $k=1$, it follows from \eqref{ujk}--\eqref{rs}  that we obtain
\begin{align*}
u_j^{(1)}&= b^{-1} \left[ 2g \partial^2_{xx} + 2g'(b-y) \partial^2_{xy} +
g''(b-y) \partial_y + 2 \kappa_j^2 g \right] \varphi_j^{(0)},\\
p^{(1)}&=  \left( b^{-1} g \partial_x + g' \partial_y \right)
\varphi_1^{(0)},\\
q^{(1)}&=- \left( b^{-1} g \partial_x + g' \partial_y \right) \varphi_2^{(0)},\\
 r_j^{(1)}&=-b^{-1} g (\mathscr T_j\varphi_j^{(0)} +
\tau_j).
\end{align*}
Using the convolution theorem and \eqref{rvp0a}--\eqref{rvp0b} yields 
\begin{subequations}\label{eq:u1}
\begin{align}
u_{jn}^{(1)}(y)&=\sum_{m \in \mathbb
Z}U_{j}^{(n,m)}(y)g_{n-m},\\
p_{1n}(y)&=\sum_{m \in \mathbb
Z}P_m(y)g_{n-m}, \\
q_{1n}(y)&=\sum_{m \in \mathbb
Z}Q_m(y)g_{n-m},\\
r_{jn}^{(1)} (y)&=-b^{-1} \sum_{m \in \mathbb Z} \left (
R_{jm}(y) +\tau_{jm}\right) g_{n-m},
\end{align}
\end{subequations}
where
\begin{align*}
U_{j}^{(n,m)}(y)&=b^{-1}\left[2 (\beta_{jm})^2 M_{j1}^{(m)}(y)  +
(\alpha_{m}^2-\alpha_n^2 )(b-y ) \partial_y  M_{j1}^{(m)}(y) \right]
\tau_{1m}\\
&\quad +b^{-1}\left[2( \beta_{jm})^2 M_{j2}^{(m)}(y)  + (\alpha_{m}^2-\alpha_n^2
)(b-y ) \partial_y  M_{j2}^{(m)}(y)   \right] \tau_{2m},\\
P_m(y)&={\rm i}\alpha_m b^{-1}\left(M_{1 1}^{(m)}(y)\tau_{1m}+M_{1
2}^{(m)}(y)\tau_{2m}\right )\\
&\quad +{\rm i}(\alpha_n-\alpha_m) \left(\partial_y M_{1
1}^{(m)}(y)\tau_{1m}+\partial_y M_{1 2}^{(m)}(y)\tau_{2m}\right),\\
Q_m(y)&=-{\rm i}\alpha_m b^{-1}\left(M_{2 1}^{(m)}(y)\tau_{1m}+M_{2
2}^{(m)}(y)\tau_{2m}\right )\\
&\quad -{\rm i}(\alpha_n-\alpha_m) \left(\partial_y M_{2
1}^{(m)}(y)\tau_{1m}+\partial_y M_{2 2}^{(m)}(y)\tau_{2m}\right )\\
R_{jm}(y)&={\rm i} \beta_{jm}\left (M_{j1}^{(m)}(y)\tau_{1m}+M_{j
2}^{(m)}(y)\tau_{2m} \right).
\end{align*}

When $k=1$, recalling the expressions of $\varphi_{jn}^{(1)}(0)$ and evaluating
\eqref{rvpk}  at $y=b$,  we have 
\begin{align*}
 \varphi_{1n}^{(1)}(b)&=K_1(b; \beta_{1n})(q_n^{(1)}+{\rm
i}\alpha_n\varphi_{2n}^{(1)}(0))-K_2(b; \beta_{1n})r_{1n}^{(1)}+\int_0^b
K_3(b, z; \beta_{1n})u_{1n}^{(1)}(z){\rm d}z \\
&=\frac{e^{{\rm i}\beta_{1n} b}}{{\rm i}\beta_{1n}}
(q_n^{(1)}+{\rm i}\alpha_n\varphi_{2n}^{(1)}(0))-\frac{e^{{\rm i}\beta_{1n}
b}}{2{\rm i}\beta_{1n}}(e^{{\rm i}\beta_{1n} b}+e^{-{\rm i}\beta_{1n}
b}) r_{1n}^{(1)}\\
&\quad +\int_0^b \frac{e^{{\rm i}\beta_{1n} b}}{2{\rm
i}\beta_{1n}}(e^{{\rm i}\beta_{1n} z}+e^{-{\rm i}\beta_{1n} z})
u_{1n}^{(1)}(z){\rm d}z \\
&=\frac{e^{{\rm i}\beta_{1n} b}}{( 2{\rm i}\beta_{1n}) (\alpha_n^2 +
\beta_{1n}\beta_{2n})}\Bigg (2 \beta_{1n}\beta_{2n}  q_n^{(1)}+ 2 
\alpha_n\beta_{1n} p_n^{(1)} - 2 \alpha_n\beta_{1n}
 e^{{\rm i}\beta_{2n} b} r_{2n}^{(1)} \\
 &\quad +( \alpha_n^2 -\beta_{1n} \beta_{2n} )e^{{\rm
i}\beta_{1n} b} r_{1n}^{(1)}-( \alpha_n^2 + \beta_{1n}\beta_{2n}) 
e^{-{\rm i}\beta_{1n} b} r_{1n}^{(1)}+ 2 \alpha_n\beta_{1n} \int_0^b e^{{\rm
i}\beta_{2n} z} u_{2n}^{(1)}(z){\rm d}z \\
& \quad -2 \alpha_n^2  \int_0^b e^{{\rm i}\beta_{1n} z} u_{1n}^{(1)}(z){\rm
d}z+( \alpha_n^2 + \beta_{1n}\beta_{2n}) \int_0^b (e^{{\rm
i}\beta_{1n} z}+e^{-{\rm i}\beta_{1n} z}) u_{1n}^{(1)}(z){\rm d}z\Bigg),
\end{align*}
and
\begin{align*}
\varphi_{2n}^{(1)}(b)&=K_1(b; \beta_{2n})(p_n^{(1)}-{\rm
i}\alpha_n\varphi_{1n}^{(1)}(0))-K_2(b; \beta_{2n})r_{2n}^{(1)}+\int_0^b
K_3(b, z; \beta_{2n})u_{2n}^{(1)}(z){\rm d}z \\
&=\frac{e^{{\rm i}\beta_{2n} b}}{{\rm i}\beta_{2n} }  (p_n^{(1)}-{\rm
i}\alpha_n\varphi_{1n}^{(1)}(0))-\frac{e^{{\rm i}\beta_{2n} b}}{2{\rm
i}\beta_{2n}}(e^{{\rm i}\beta_{2n} b} +e^{-{\rm i}\beta_{2n} b}) r_{2n}^{(1)}\\
&\quad +\int_0^b  \frac{e^{{\rm i}\beta_{2n} b}}{2{\rm
i}\beta_{2n}}(e^{{\rm i}\beta_{2n} z}+e^{-{\rm i}\beta_{2n} z})
u_{2n}^{(1)}(z){\rm d}z \\
&=\frac{e^{{\rm i}\beta_{2n} b}}{( 2{\rm i}\beta_{2n}) ( \alpha_n^2 +
\beta_{1n}\beta_{2n})}\Bigg (2\beta_{1n}\beta_{2n}p_n^{(1)}-
2\alpha_n \beta_{2n} q_n^{(1)} +  2\alpha_n \beta_{2n} e^{{\rm
i}\beta_{1n} b} r_{1n}^{(1)}\\
&\quad  + (\alpha_n^2- \beta_{1n}\beta_{2n})  e^{{\rm i}\beta_{2n} b}
r_{2n}^{(1)}- ( \alpha_n^2 + \beta_{1n}\beta_{2n} ) e^{-{\rm
i}\beta_{2n}b} r_{2n}^{(1)}- 2 \alpha_n \beta_{2n}\int_0^b e^{{\rm
i}\beta_{1n} z} u_{1n}^{(1)}(z){\rm d}z  \\
&\quad -2 \alpha_n^2 \int_0^b e^{{\rm i}\beta_{2n} z} u_{2n}^{(1)}(z){\rm
d}z + ( \alpha_n^2 + \beta_{1n}\beta_{2n} )\int_0^b (e^{{\rm
i}\beta_{2n} z}+e^{-{\rm i}\beta_{2n} z}) u_{2n}^{(1)}(z){\rm d}z\Bigg).
\end{align*}

Substituting  \eqref{eq:u1} into \eqref{rvpk} and evaluating at $y=b$, after
tedious  but straight forward calculations,  we obtain the key identities:
\begin{subequations}\label{rvpk1}
\begin{align}
 \varphi_{1n}^{(1)}(b)=&\sum_{m \in \mathbb Z} \frac{e^{{\rm i}\beta_{1n}
b}}{( 2{\rm i}\beta_{1n})(\alpha_n^2 + \beta_{1n}\beta_{2n})( \alpha_m^2 +
\beta_{1m}\beta_{2m})}A_1^{(n,m)} g_{n-m} ,\\
 \varphi_{2n}^{(1)}(b)=&\sum_{m \in \mathbb Z} \frac{e^{{\rm i}\beta_{2n}
b}}{( 2{\rm i}\beta_{2n})(\alpha_n^2 + \beta_{1n}\beta_{2n})( \alpha_m^2 +
\beta_{1m}\beta_{2m})}A_2^{(n,m)} g_{n-m}, 
 \end{align}
 \end{subequations}
where
\begin{align*}
A_{1}^{(n,m)} =&\Bigg \{ b^{-1}\bigg [ -2\beta_{1n} \beta_{2n}  \alpha_m^2
e^{ {\rm i}(\beta_{1m}+\beta_{2m})b} + \frac{\alpha_n \alpha_m
\beta_{1n}}{\beta_{1m}} ( \alpha_m^2 - \beta_{1m} \beta_{2m} ) e^{2{\rm i}
\beta_{1m}b}\\
&\quad + 2\Big \{ \alpha_m \beta_{1n} (\alpha_n \beta_{2m}+ \alpha_m
\beta_{2n}) + {\rm i} b \beta_{1n} \big [ \alpha_n \alpha_m 
\beta_{2m}( \beta_{2n}-  \beta_{2m})  \\
&\quad -( \alpha_n \alpha_m )^2  + \beta_{1m}^2  \beta_{2m}
\beta_{2n} \big  ]\Big \} e^{{\rm i}  \beta_{1m} b} - \frac{\alpha_n 
\alpha_m \beta_{1n} }{\beta_{1m}} (\alpha_m^2+ \beta_{1m} \beta_{2m} )\bigg ]\\
&\quad -   {\rm i} \beta_{1n}   (\alpha_n -\alpha_m)   \bigg [ 2 \alpha_m
\beta_{2m} \beta_{2n}  e^{{\rm i} (\beta_{1m}+ \beta_{2m} ) b}-
\alpha_n  (\alpha_m^2 -\beta_{1m} \beta_{2m} ) e^{2{\rm i} \beta_{1m} b}\\
&\quad -\alpha_n (\alpha_m^2 +\beta_{1m}  \beta_{2m} )  \bigg ] \Bigg \} 
\tau_m^{(1)} % end of tau_1
+\Bigg \{  b^{-1}\bigg [-2 \alpha_n \alpha_m^2 \beta_{1n}  e^{{\rm
i}(\beta_{1m} + \beta_{2m}) b }\\
&\quad  - \frac{\alpha_m \beta_{1n} \beta_{2n}}{ \beta_{2m}} ( \alpha_m^2
-\beta_{1m}  \beta_{2m} ) e^{2 {\rm i} \beta_{2m} b}+ 2 \Big \{  \alpha_m
\beta_{1n} \big( \alpha_n  \alpha_m - \beta_{1m}  \beta_{2n} \big)\\
&\quad +  {\rm i} b \beta_{1n} \Big [ 	\alpha_n (	\alpha_m^2 
\beta_{2n}+ \beta_{2m}^2 \beta_{1m})+ \alpha_m \beta_{1m}  (\alpha_n^2
+\beta_{1m}  \beta_{2n})  \Big ] \Big \}e^{ {\rm i} \beta_{2m} b}\\
&\quad + \frac{\alpha_m \beta_{1n}  \beta_{2n}}{\beta_{2m}}( \alpha_m^2 +
\beta_{1m}  \beta_{2m} )  \bigg ] -{\rm i} \beta_{1n}
(\alpha_n-\alpha_m) \bigg[  2 \alpha_n \alpha_m  \beta_{1m} e^{{\rm i}
(\beta_{1m} + \beta_{2m} )b }\\
&\quad + \beta_{2n} (\alpha_m^2 - \beta_{1m} \beta_{2m}) e^{2 {\rm
i} \beta_{2m} b} + \beta_{2n} (\alpha_m^2+ \beta_{1m}  
\beta_{2m})   \bigg] \Bigg \} \tau_m^{(2)},
\end{align*}
and
\begin{align*}
A_{2}^{(n,m)} =&\Bigg \{  b^{-1}\bigg [  2\alpha_{n}
\alpha_m^2  \beta_{2n}   e^{ {\rm i} (\beta_{1m}+ \beta_{2m}  ) b} +
\frac{ \alpha_m \beta_{1n} \beta_{2n}  }{\beta_{1m} } ( \alpha_m^2
-\beta_{1m} \beta_{2m}  ) e^{2{\rm i} \beta_{1m} b}\\
&\quad -2\Big \{  \alpha_m  \beta_{2n} (  \alpha_n  \alpha_m -\beta_{1n} 
\beta_{2m} ) + {\rm i} b   \beta_{2n} \Big[ \alpha_n   (\alpha_m^2  
\beta_{1n} +\beta_{1m}^2 \beta_{2m})\\
&\quad +  \alpha_m \beta_{2m}  \big( \alpha_n^2+  \beta_{1n}
\beta_{2m} )  \Big] \Big \}  e^{{\rm i}  \beta_{1m} b}- \frac{  \alpha_m
\beta_{1n}  \beta_{2n} }{\beta_{1m} } (\alpha_m^2+ \beta_{1m}
\beta_{2m} ) \bigg ] \\
&\quad + {\rm i} \beta_{2n}  (\alpha_n -\alpha_m)  \Big [ 2 \alpha_n
\alpha_m   \beta_{2m}   e^{{\rm i}(\beta_{1m} + \beta_{2m})  b}+
\beta_{1n}  (  \alpha_m^2- \beta_{1m} \beta_{2m} ) e^{2{\rm i}
\beta_{1m} b}\\
&\quad +   \beta_{1n} (\alpha_m^2 +\beta_{1m} \beta_{2m}) \Big]
\Bigg \} \tau_m^{(1)}+\Bigg \{ b^{-1}\bigg [  -2 \beta_{1n} \beta_{2n}
\alpha_m^2  e^{{\rm i}(\beta_{1m} + \beta_{2m})b} \\
&\quad + \frac{\alpha_m \alpha_n \beta_{2n}}{ \beta_{2m}} (
\alpha_m^2- \beta_{1m}  \beta_{2m}) e^{2 {\rm i} \beta_{2m}b}+  2
\Big  \{ \alpha_m \beta_{2n} \big (\alpha_n \beta_{1m} +\alpha_m
\beta_{1n} \big) \\
&\quad +{\rm i} b  \beta_{2n} \Big [\alpha_n \alpha_m  \beta_{1m} (
\beta_{1n} -\beta_{1m})  -(\alpha_n \alpha_m)^2 + \beta_{2m}^2
\beta_{1m} \beta_{1n}  \Big] \Big   \}   e^{ {\rm i} \beta_{2m}b }\\
&\quad - \frac{\alpha_m \alpha_n \beta_{2n}}{ \beta_{2m}} (
\alpha_m^2 + \beta_{1m}  \beta_{2m} )  \bigg ] -{\rm i} \beta_{2n}
(\alpha_n-\alpha_m) \bigg[ 2 \alpha_m  \beta_{1n}   \beta_{1m}  e^{{\rm i}
(\beta_{1m} + \beta_{2m})b}\\
&\quad  - \alpha_n (\alpha_m^2 - \beta_{1m}\beta_{2m}) e^{2 {\rm i}
\beta_{2m}  b} -\alpha_n   (\alpha_m^2+ \beta_{1m}  \beta_{2m})  
\bigg ] \Bigg \}\tau_m^{(2)}.
\end{align*}

\section{inverse problem}\label{sec:ip}

In this section, we give reconstruction formulas for the inverse problem by
dropping the higher order terms in the power series. Moreover, a nonlinear
correction scheme is proposed to improve the accuracy of the reconstruction.

\subsection{Reconstruction formula}

First, we rewrite the power series expansion \eqref{ps} of $\varphi_1$ and
$\varphi_2$ as follows, 
\begin{equation}\label{eq:varphi}
\varphi_j(x,y)=\varphi_j^{(0)}(x,y)+\varepsilon \varphi_j^{(1)}(x,y)+e_j(x,y),
\end{equation}
where $e_j(x,y)={\mathcal O}(\varepsilon^2)$ denote the remainder consisting of
all the high oder terms. Evaluating \eqref{eq:varphi} at $y=b$ and dropping
$e_j(x,y)$, we get the linearized equation:
\[
\varphi_j(x,b)=\varphi_j^{(0)}(x,b)+\varepsilon \varphi_j^{(1)}(x,b),
\]
which, in the frequency domain, 
\begin{equation}\label{eq:sec 7}
\varphi_{jn}(b)=\varphi_{jn}^{(0)}(b)+\varepsilon \varphi_{jn}^{(1)}(b).
\end{equation}
Substituting \eqref{rvpk1} into \eqref{eq:sec 7} and noting $f=\varepsilon g$,
we obtain an infinite dimensional linear system of equations:
\[
\sum_{m \in \mathbb Z} C_j^{(n, m)} f_{n-m}
= \varphi_{jn}(b)-\varphi_{jn}^{(0)}(b), 
\]
where
\[
 C_j^{(n, m)}=\frac{e^{{\rm i}\beta_{jn} b}}{( 2{\rm i}\beta_{jn})
( \alpha_n^2 + \beta_{1n}\beta_{2n}) (\alpha_m^2 +
\beta_{1m}\beta_{2m})}A_j^{(n,m)}.
\]

In order to obtain a truncated finite dimensional linear systems, the
cut-off 
\[
N_j=\left \lfloor\frac{\eta_j \Lambda} {2\pi } \right \rfloor
\]
is chosen such that $|\alpha_n| \leq \eta_j$ for all $|n| \leq N_j$, where
$\eta_j$ is given by \eqref{eq:eta}. In view of the definition of $\eta_j$, the
density $\rho_1$ of the elastic slab is crucial to the reconstruction
resolution, a bigger $\rho_1$ gives a higher resolution. Keeping only the
Fourier coefficients of the solution in $[-N_j, \, N_j]$, we obtain the
truncated equations 
\begin{equation}\label{slt}
C_j s_j=t_j,
\end{equation}
where $C_j$ is the $(2N_j+1) \times (2N_j+1)$ portion of $C_j^{(n, m)}$, 
and $s_j, t_j$ are $(2N_j+1)$ column vectors given by
\[
s_{j, m}=f_m,\quad t_{j, n}=  \varphi_{jn}(b)-\varphi_{jn}^{(0)}(b),\quad -N_j
\leq n, m\leq N_j.
\]
We observe from \eqref{tau} and \eqref{rvpk1} that when $|m| >N_j$ there could
have  exponentially amplified errors of $A_j^{(n,m)}$ due to the data noise.  
Therefore, the equations need to be regularized further by letting
$A_j^{(n,m)}=0$ if $|n-m| >N_j$. Let the solution of \eqref{slt}
 be given by 
\begin{equation}\label{pdi}
s_j=C_j^\dagger t_j,
\end{equation}
where $C_j^\dagger$ denote the Moore-Penrose pseudo-inverse of $C_j$. Finally,
the scattering surface function is reconstructed as follows: 
\begin{equation}\label{reconstruct}
f(x)=\mathrm {Re} \sum_{|m| \leq N_j} s_{j,m} e^{{\mathrm i} \alpha_m x}.
\end{equation}

\subsection{Nonlinear correction scheme}

In the previous subsection, an explicit reconstruction formula
\eqref{reconstruct} is given. It is effective for a sufficiently
small deformation parameter $\varepsilon$. For a relatively large
$\varepsilon$, it is necessary to develop a nonlinear correction scheme to
improve the accuracy of the reconstruction.

Firstly, we solve the linearized problem and compute \eqref{pdi} to obtain
$s_j$, which is denoted as $s_j^{[0]}$. Let $f_0$ be the reconstructed
surface function by using $s_j^{[0]}$ in \eqref{reconstruct}. Next we solve the
direct problem using $f_0$ as the surface function, and
evaluate the total field $\boldsymbol{u}$ at $y=a$ denoted by
$\boldsymbol{u}^{[f_0]}$. The data $\phi_j^{[f_0]}(x,a)$ is computed from
\eqref{phin} by using $\boldsymbol{u}^{[f_0]}$, which is then used to compute
$\tau_{jn}^{[f_0]}$ from \eqref{fvps}, \eqref{pfvpb}
and \eqref{tau}. We construct the coefficient matrices $C_j^{[f_0]}$ and
the right hand side vectors $t_j^{[f_0]}$ of \eqref{slt} using
$\tau_{jn}^{[f_0]}$. Now we have approximated equations:
\[
C_j^{[f_0]} s_{j}^{[0]}=t_j^{[f_0]}.
\]
Subtracting the above equation from \eqref{slt} yields
\[
C_j s_j=t_j+C_j^{[f_0]} s_j^{[0]}-t_j^{[f_0]},
\]
from which we compute the updated Fourier coefficients:
\[
 s_{j}^{[1]}=C_j^\dagger \left(t_j+C_j^{[f_0]} s_j^{[0]}-t_j^{[f_0]} \right). 
\]
Then the surface function is updated as follows
\[
f_1(x)={\rm Re} \sum_{|m| \leq N_j} s_{j,m}^{[1]} e^{{\rm i} \alpha_m x}.
\]
Repeating the above procedure gives the nonlinear correction scheme:
\begin{align*}
&s_{j}^{[l]}=C_j^\dagger \left(t_j+C_j^{[f_{l-1}]}
s_{j}^{[l-1]}-t_j^{[f_{l-1}]} \right),\\
&f_l(x)={\rm Re} \sum_{|m| \leq
N_j} s_{j,m}^{[l]} e^{{\rm i} \alpha_m x},  \quad l=1,\dots.
\end{align*}

Essentially the above nonlinear correction scheme is similar to Newtown's
method for solving non-linear equations. From the numerical experiments in the
next section, we only need few iterations to obtain accurate reconstructions
because good initial guesses are available from the reconstruction formula
\eqref{reconstruct} when solving the linearized equation.

\section{Numerical experiments}\label{sec:ne}

In this section, we present some numerical experiments to show the
effectiveness of the proposed method. We solve the direct scattering problem
\eqref{tf} to get the  synthetic  data of the displacement of the total field
$\boldsymbol{u}$ by using the finite element method with the perfectly matched
layer (PML) technique. Then the measured data is obtained by interpolating the
finite element solution with $500$ uniform grid on  $\Gamma_a$. In order to test
the robustness of  the proposed method, we add random noise to the data:
\[
\boldsymbol{u}_{\delta}(x_i,a)=\boldsymbol{u}(x_i,a)(1+\delta \boldsymbol
r_i),
\]
where $x_i=-\Lambda/2+i\Lambda/500, i=1, \dots, 500,$ $\boldsymbol r_i$ are
vectors whose two components are random numbers uniformly distributed on
$[-1,\,1]$, and $\delta$ is the noise level.

In our numerical experiments, the Lam\'{e} parameters $\mu, \lambda$ are taken
as $\lambda=2, \mu=1$. The density $\rho_0$ of the free space is $\rho_0=1$,
while the density of the elastic slab $\rho_1$ is chosen to be three different
numbers $\rho_1=1.0, \, 2.0$ and $4.0$ in order to compare the reconstruction
results. The noise level $\delta=2\%$. The angular frequency $\omega=2\pi
$. Thus the compressional wavenumber $\kappa_1=\pi $ and the shear 
wavenumber $\kappa_2 =2 \pi $, which indicate that $\lambda_1=2, \lambda_2
=1$, where $\lambda_1$ and $\lambda_2$ are the compressional wavelength  and the
shear wavelength, respectively. The bottom of the slab is positioned at $y = b
= 0.05\lambda_2 $ and the top of the slab is put at $y = a = 2.0\lambda_2
$. Hence the slab is put in the near-field regime while the
data is measured in the far-field regime.  The incident wave  is generated by
\eqref{if}. In all numerical examples, the deformation parameter is fixed at
$\varepsilon=0.01$. According to \eqref{reconstruct}, there are two possible
choices to obtain the reconstructed surface function $f$, which are
mathematically equivalent. Thus we always take $j=1$ in \eqref{slt} to compute
the Fourier coefficients and to reconstruct the surface.

Example 1. The exact surface profile function is given by
\[
g(x)=\frac{1}{5} \sin\left ( \frac{20\pi x}{31}\right)-\sin\left ( \frac{40\pi
x}{31}\right)+\sin\left ( \frac{60\pi x}{31}\right),
\]	
which is a periodic function with the period $\Lambda=3.1$. This is a simple
example as the surface function only contains a few Fourier modes.

Figure \ref{fig:ex1} shows the reconstructed surfaces (dashed line) against the
exact surface (solid line). Figure \ref{fig:ex1}(a), (b), and (c) plot the
reconstructed surfaces by using $\rho_1 = 1.0, 2.0, 4.0$, respectively. Clearly,
the reconstruction resolution is increased  with respect to $\rho_1$.  For
$\rho_1 = 1.0$, the slab is absent and the cut-off $N_1= 1$. Hence
only the zeroth and first Fourier modes may be reconstructed and the
resolution is at most one wavelength. More frequency modes are able to be
recovered and the resolution increases to the subwavelength regime by increasing
$\rho_1$. Using Figure \ref{fig:ex1}(c) as the initial guess, we adopt the
nonlinear correction scheme to improve the reconstruction accuracy. As shown in
Figure \ref{fig:ex1}(d), (e), and (f), the reconstruction is almost perfect
after 3 steps of the iteration, which indicates that the algorithm is effective
to improve the accuracy of the reconstruction.

\begin{figure}
 \centering
 \includegraphics[width=0.49\textwidth]{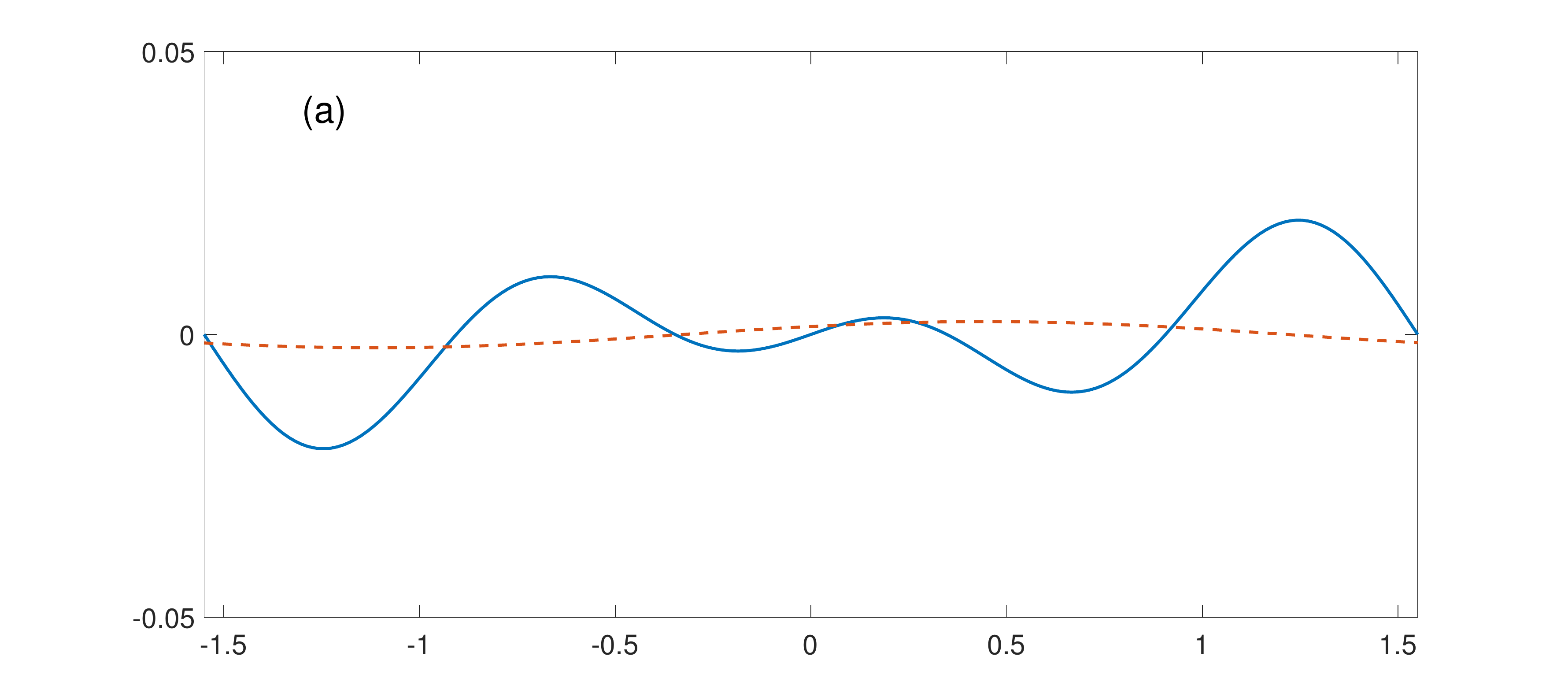}
 \includegraphics[width=0.49\textwidth]{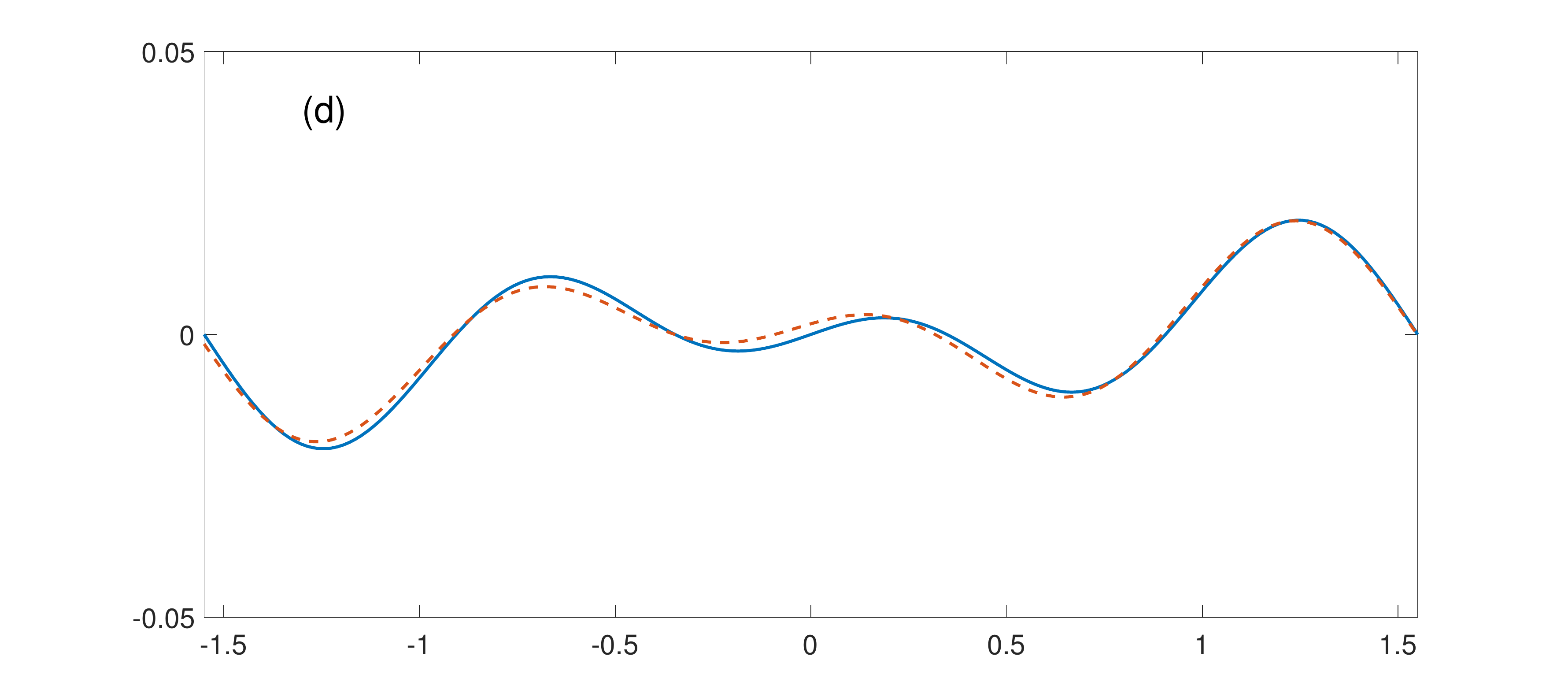}
 \includegraphics[width=0.49\textwidth]{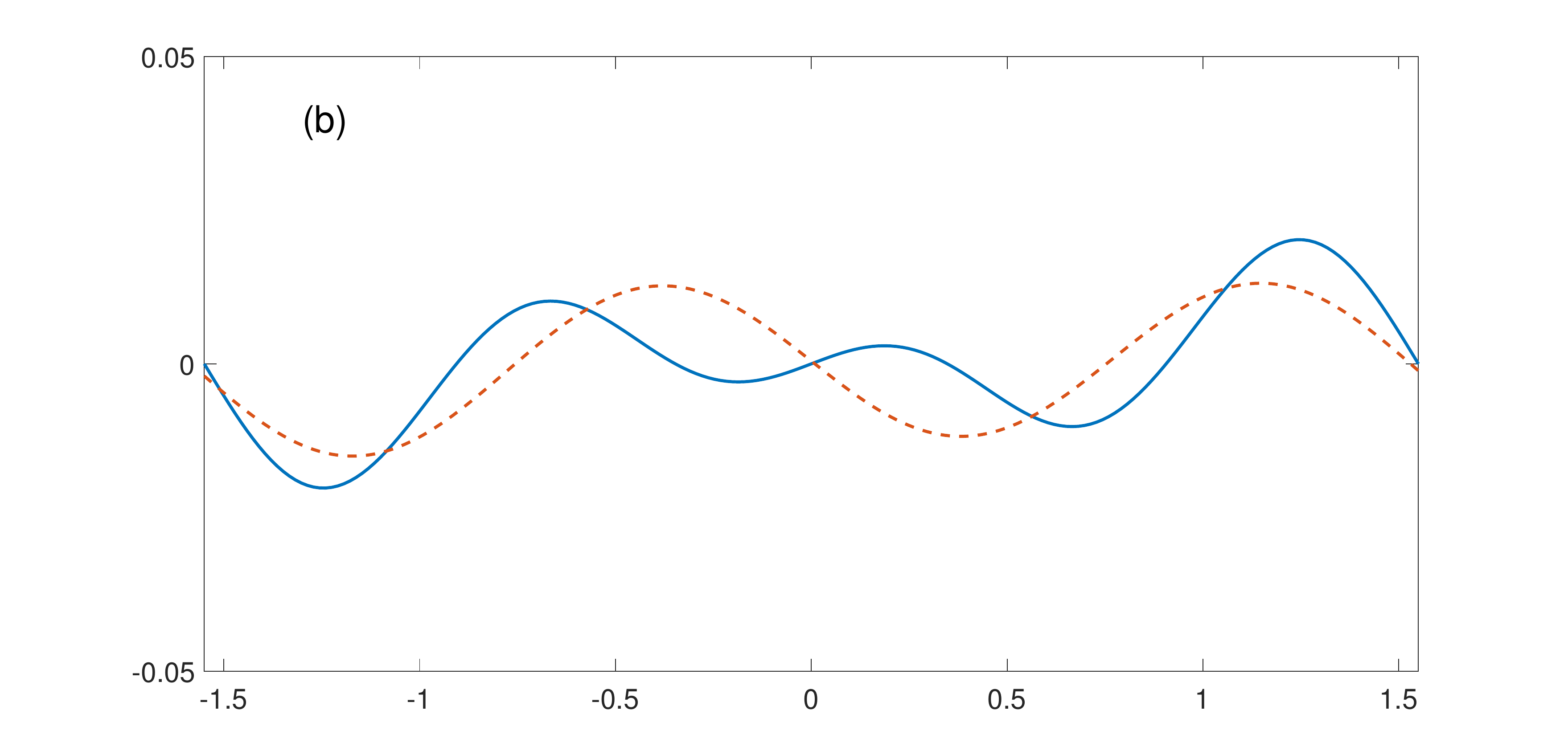}
 \includegraphics[width=0.49\textwidth]{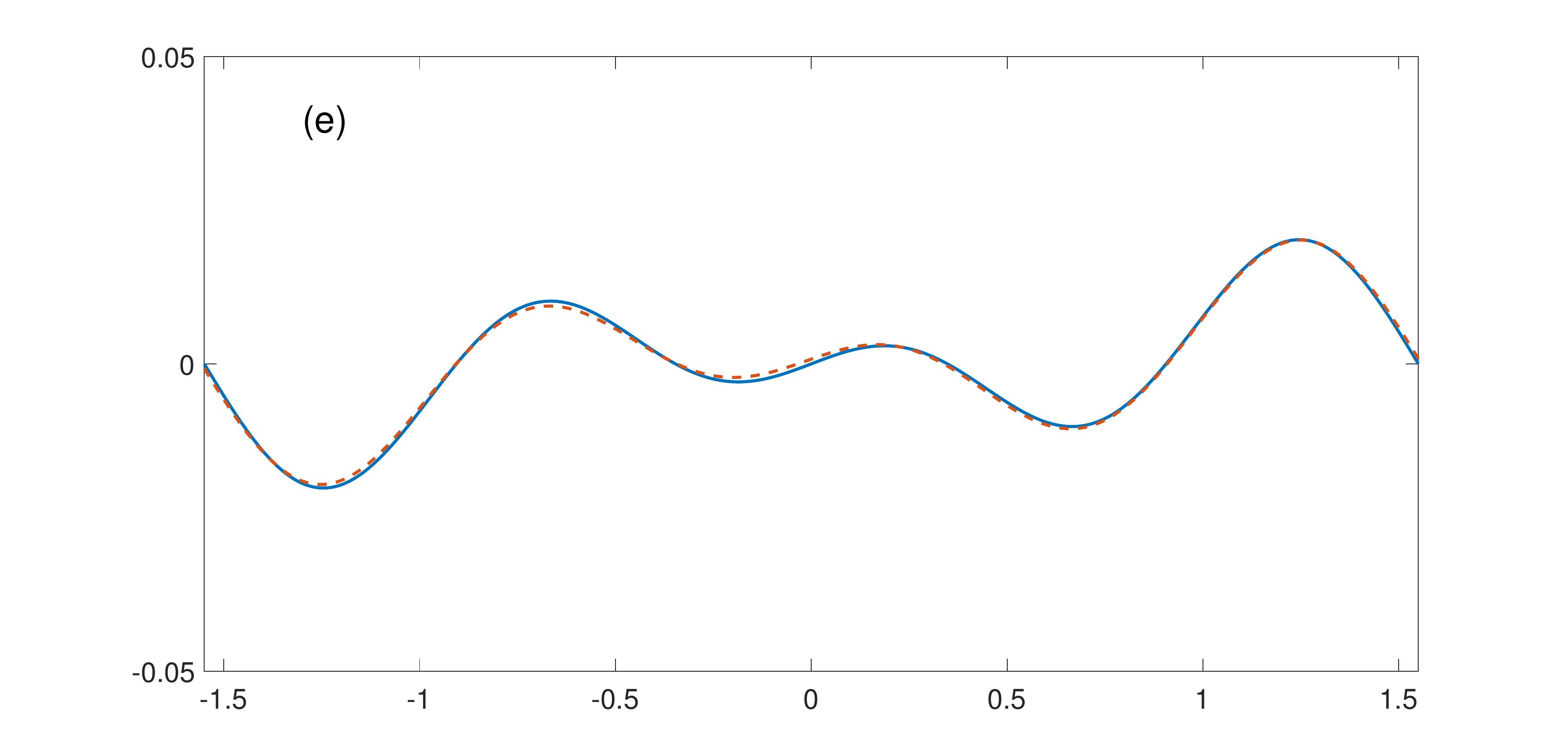}
 \includegraphics[width=0.49\textwidth]{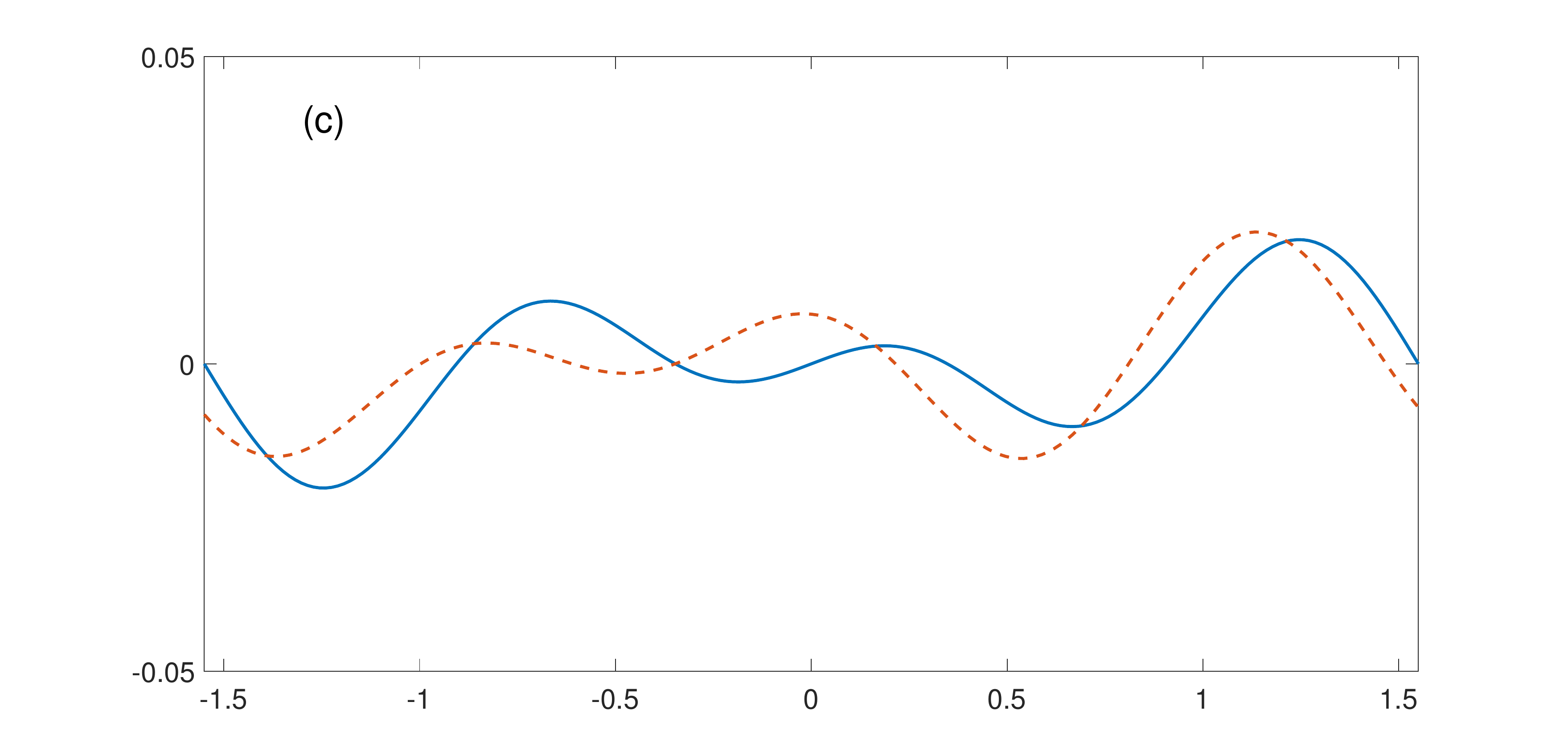}
 \includegraphics[width=0.49\textwidth]{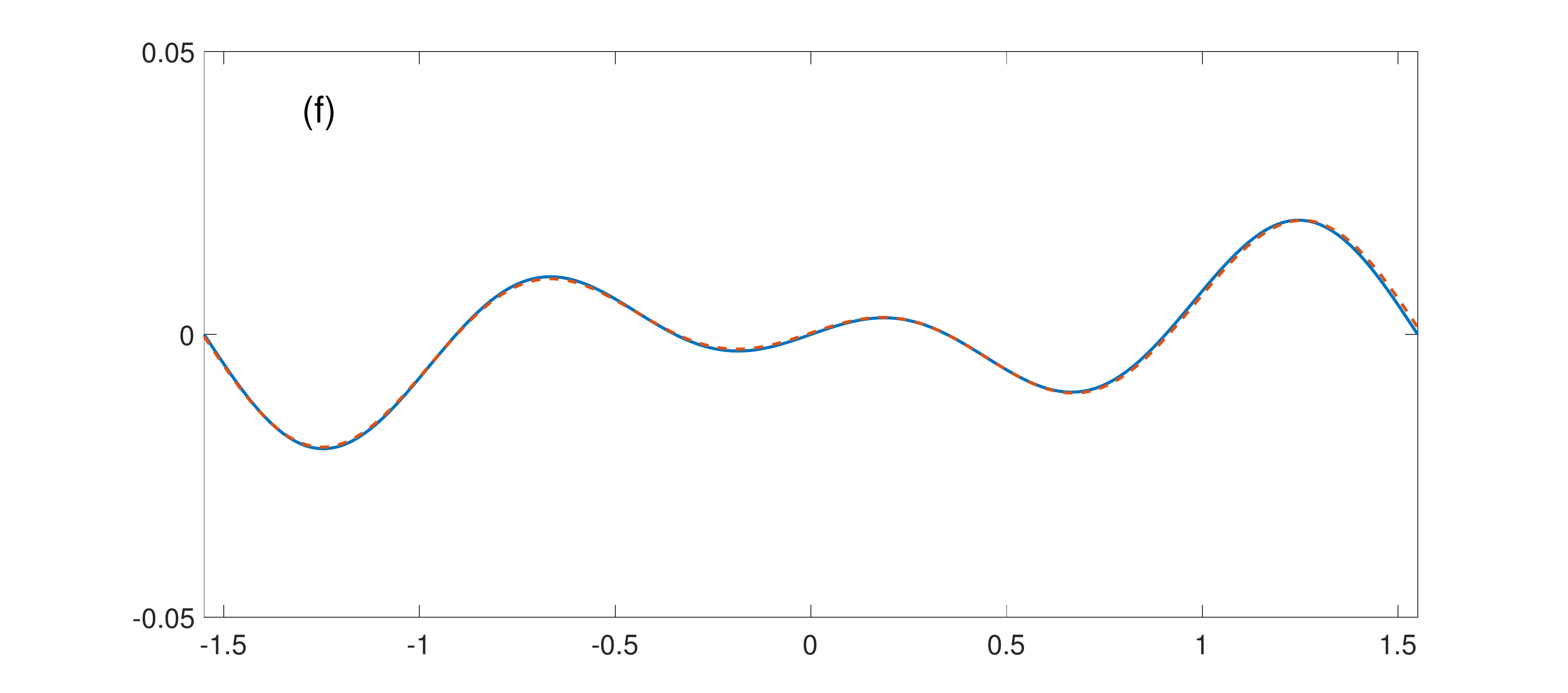}
 \caption{Example 1: the reconstructed surface (dashed line) is plotted against
the exact surface (solid line). (a) $\rho_1=1$; (b) $\rho_1=2$; (c) $\rho_1=4$;
(d) 1 step of nonlinear correction when $\rho_1=4$; (e) 2 steps of nonlinear
correction when $\rho_1=4$; (f) 3 steps of nonlinear correction when
$\rho_1=4$.}
\label{fig:ex1}
\end{figure}

Example 2. Consider the following surface profile function in the interval
$[-1, 1]$:
\[
g(x)=\begin{cases}
1- \cos (2\pi x) ,\quad& -1 \leq x< 0, \\
0.5-0.5 \cos (2\pi x),\quad& 0 < x \leq 1.\\
    \end{cases}
\]
The period $\Lambda=2$. Although this function is continuous, it is not
smooth since the first derivative is not continuous at $x=0$. Figure
\eqref{fig:ex2} shows the reconstructed surface (dashed line) against the exact
surface (solid line) for different density $\rho_1$ and the first three
steps of the nonlinear correction. The similar conclusions can be drawn
as those for Example 1: the density $\rho_1$ helps the resolution and the
nonlinear correction improve the reconstruction.

\begin{figure}
 \centering
 \includegraphics[width=0.3\textwidth]{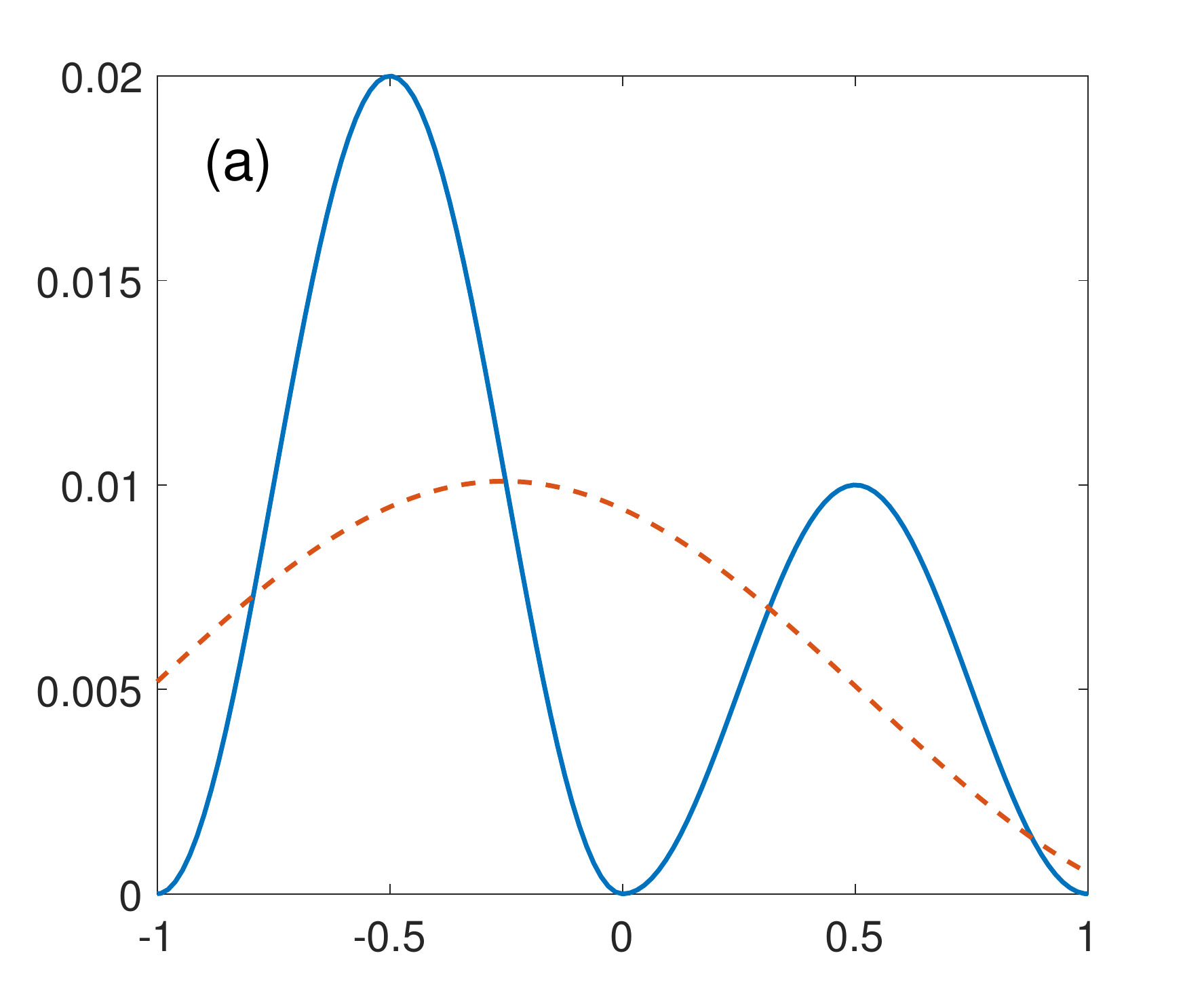}
 \includegraphics[width=0.3\textwidth]{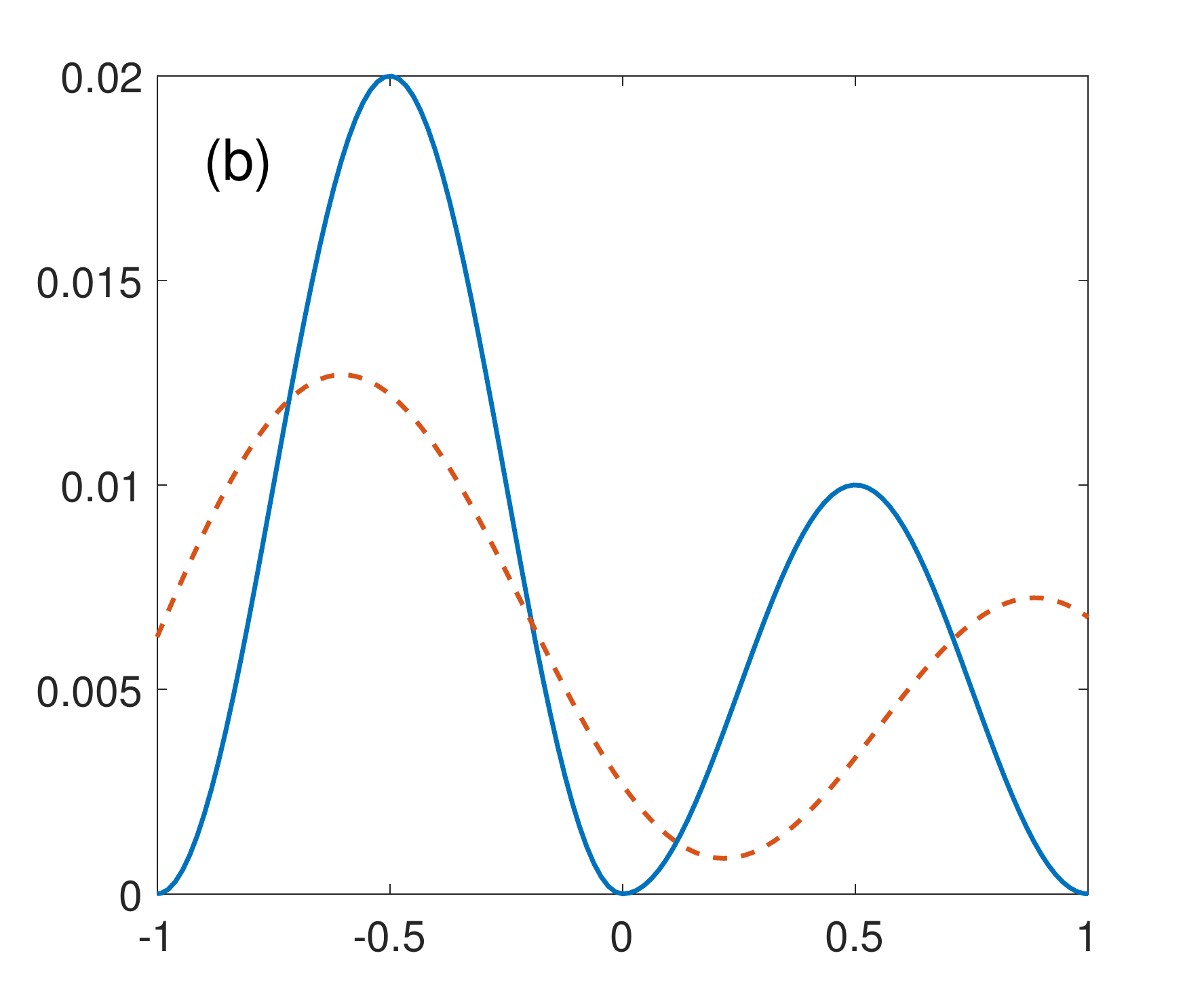}
  \includegraphics[width=0.3\textwidth]{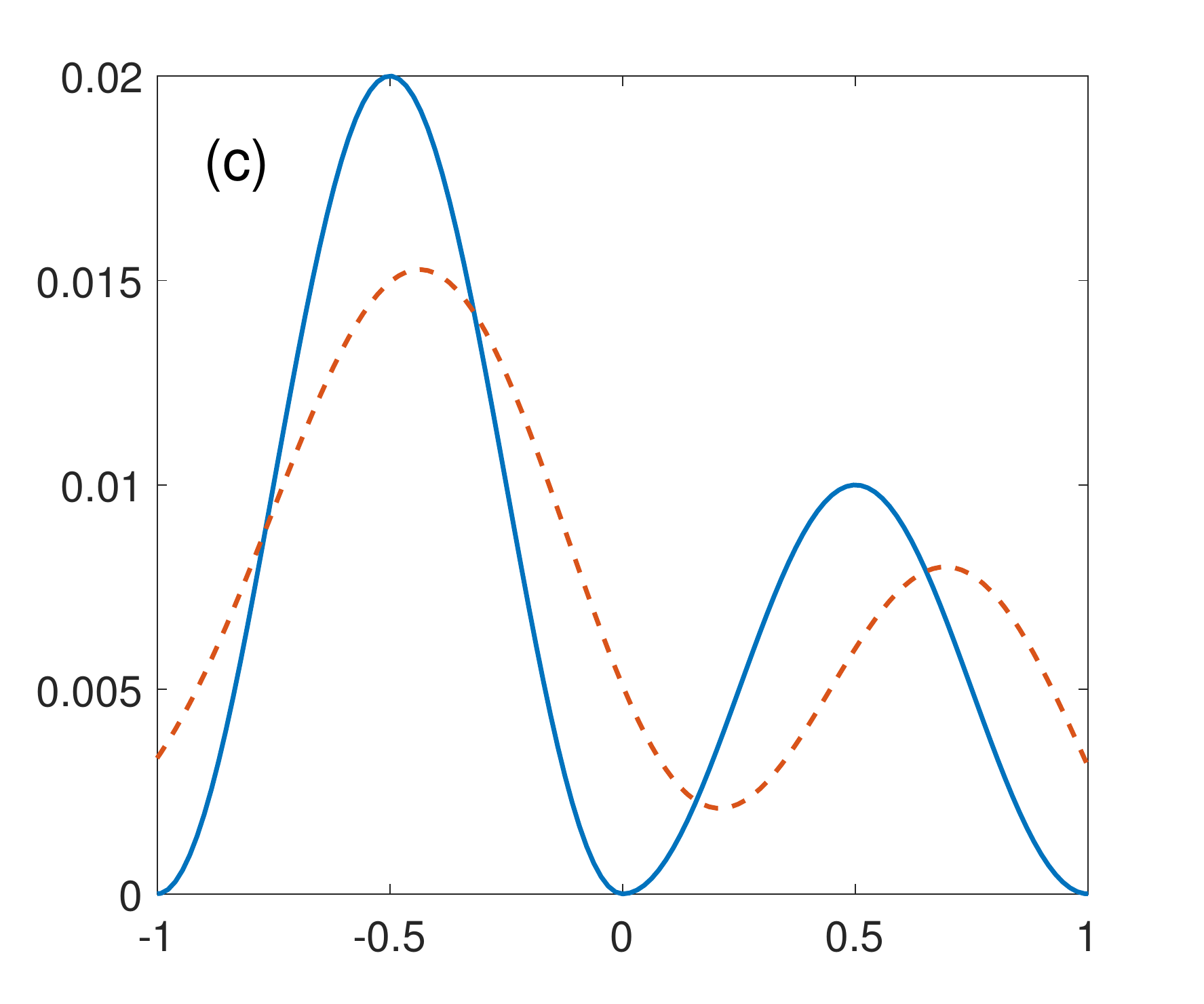}
  \includegraphics[width=0.3\textwidth]{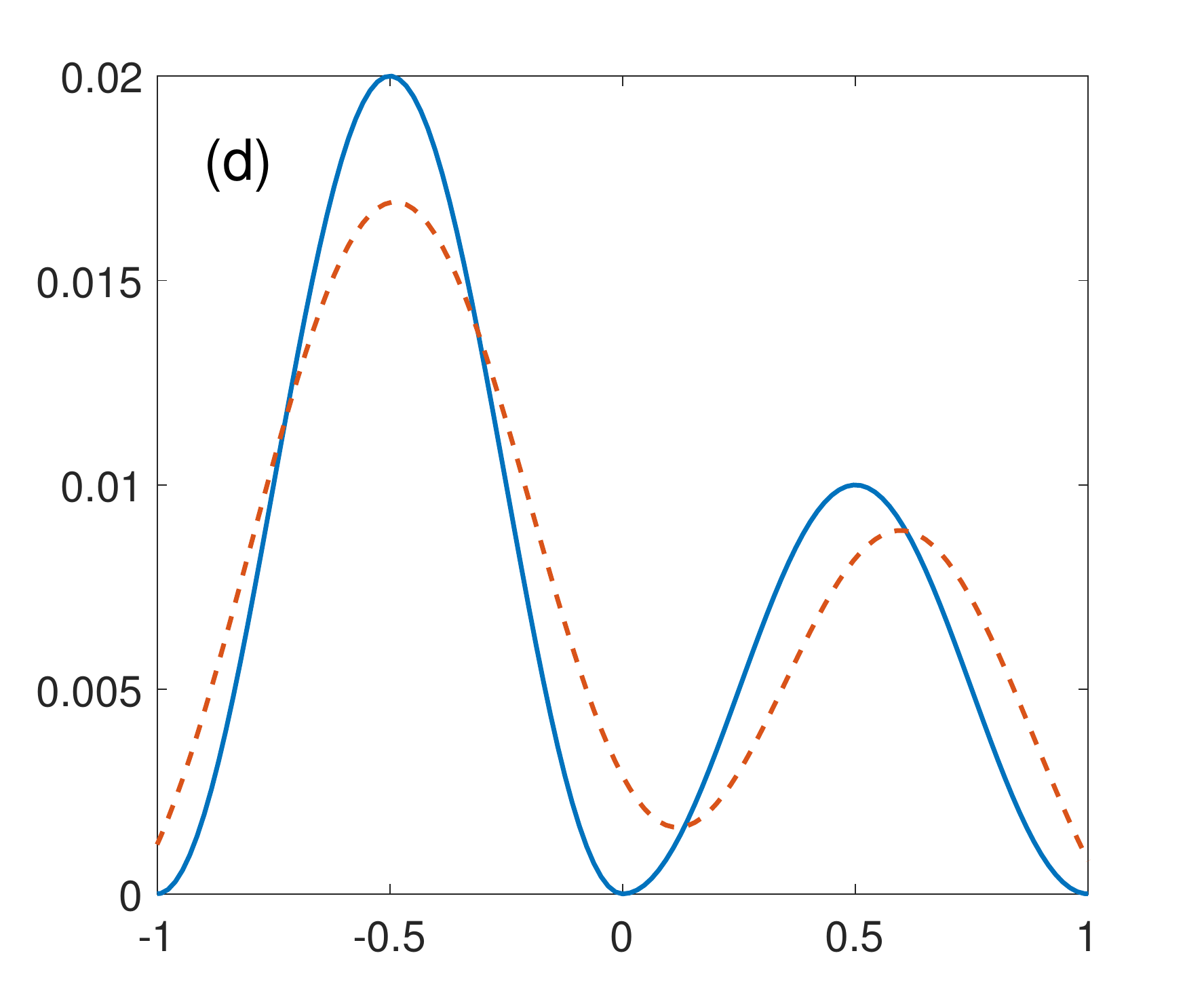}
   \includegraphics[width=0.3\textwidth]{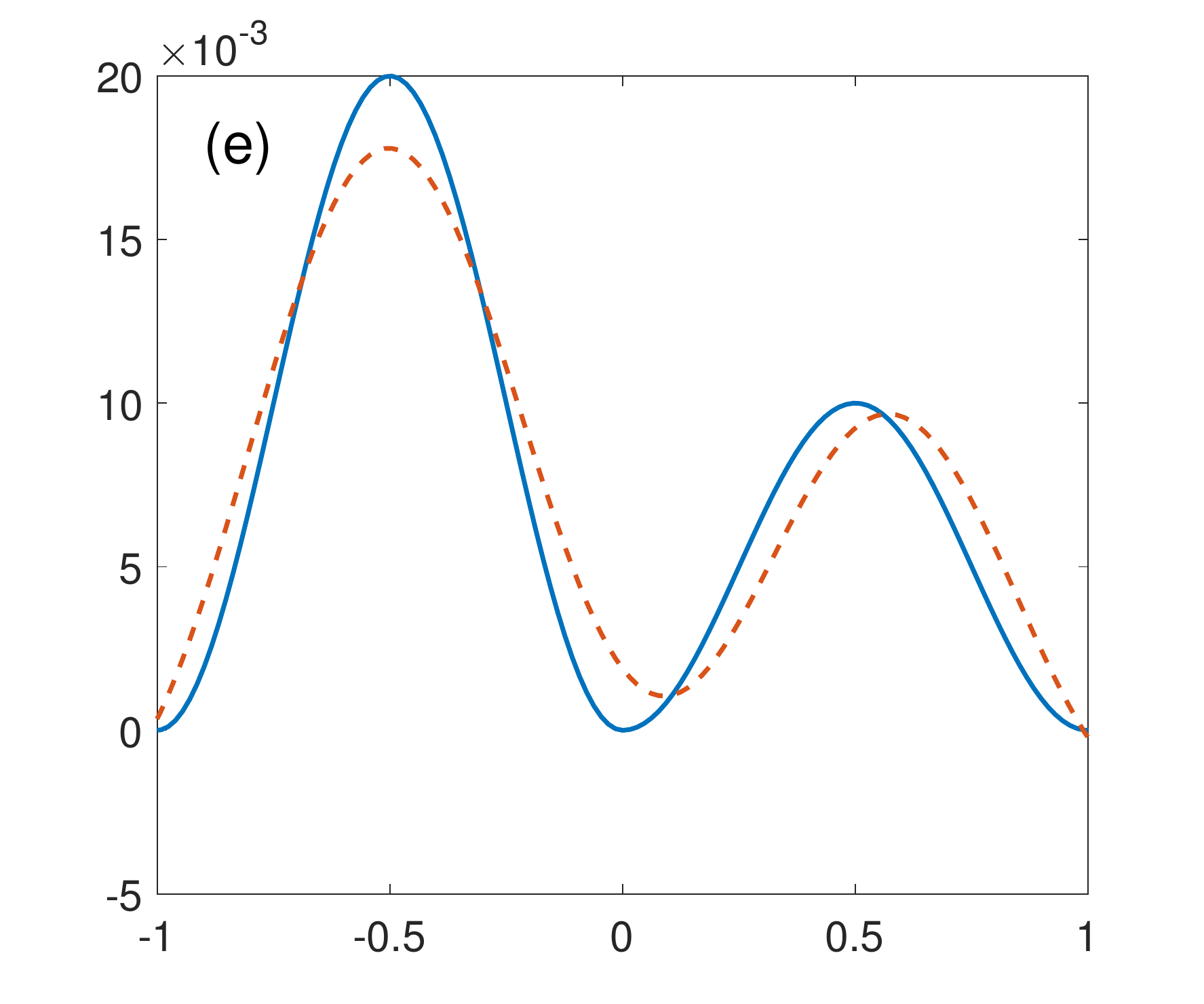}
 \includegraphics[width=0.3\textwidth]{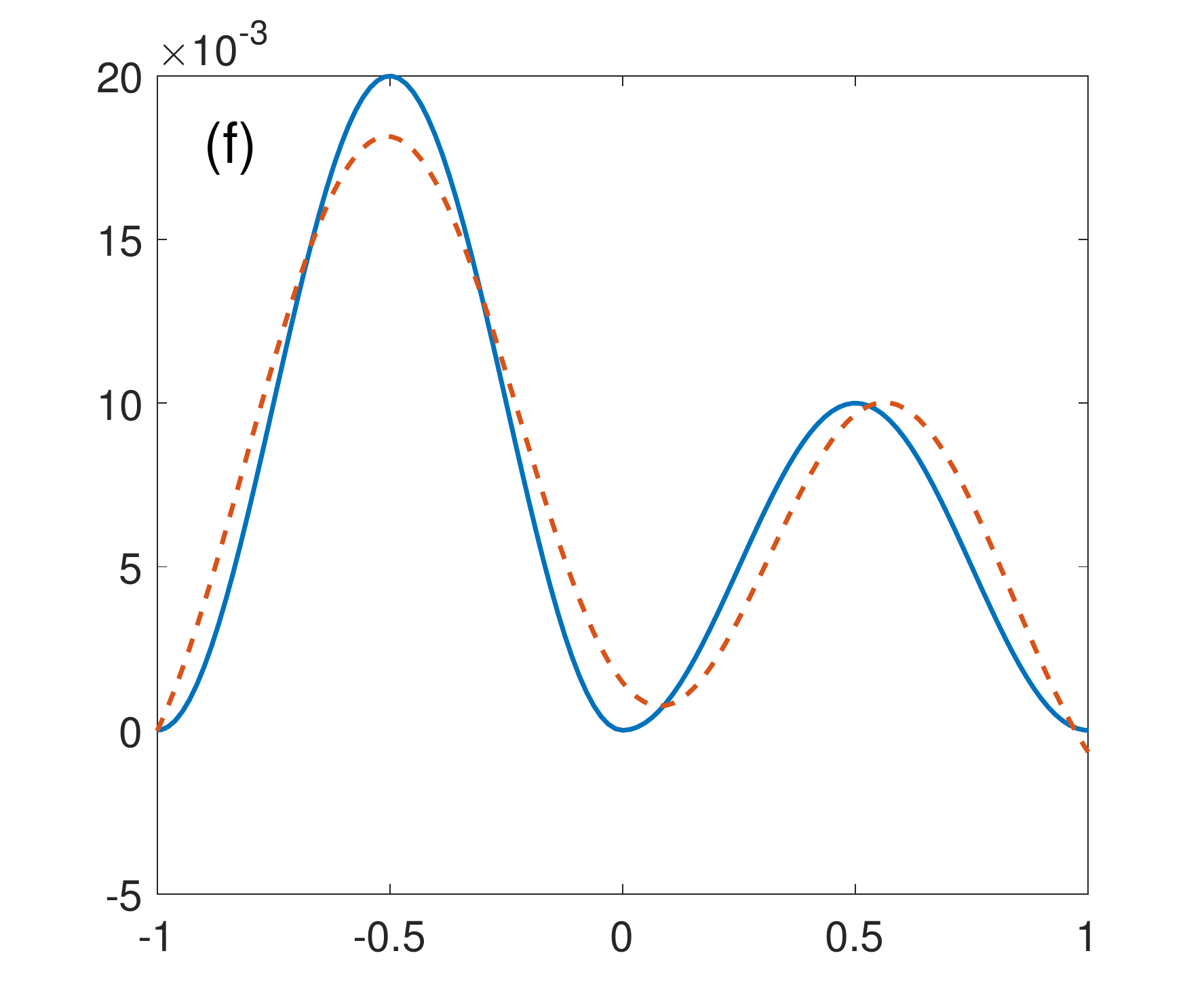}
 \caption{Example 2: the reconstructed surface (dashed line) is plotted against
the exact surface (solid line). (a) $\rho_1=1$; (b) $\rho_1=2$; (c) $\rho_1=4$;
(d) 1 step of nonlinear correction when $\rho_1=4$; (e) 2 steps of nonlinear
correction when $\rho_1=4$; (f) 3 steps of nonlinear correction when
$\rho_1=4$.}
\label{fig:ex2}
\end{figure}

\section{Conclusion}\label{sec:con}

In this paper, we have proposed an effective mathematical model and developed an
efficient numerical method to solve the inverse elastic surface scattering
problem by using the far-field data. The key idea is to utilize a slab with
larger density to allow more propagating modes to propagate to the far-field
zone, which contributes to the reconstruction resolution. The nonlinear
correction improves the accuracy by using the initial guess generated from the
explicit reconstruction formula. Results show that the proposed method is
robust to the data noise. The proposed approach can be extended to
bi-periodic structures where the three-dimensional Maxwell and elastic
equations should be considered. We are investigating these equations and will
report the progress elsewhere.

\appendix

\section{second order equations}

Consider the final value problem of the second order equation in the interval
$(b, a)$:
\begin{subequations}\label{Afvp}
\begin{align}
\label{Afvp1}  u'' + \eta^2 u = 0, \quad & b < y < a,\\
\label{Afvp2}  u = p, \quad & y=a, \\
\label{Afvp3}  u'-{\rm i}\beta u = q, \quad & y =a,
\end{align}
\end{subequations}
where $0\neq \eta, \beta, p, q$ are constants.  

\begin{lemm}\label{A1}
The final value problem (\ref{Afvp}) has a unique solution which is given by
\[
u(y)=\left(\frac{(\eta+\beta)p-{\rm i}q}{2\eta}\right)e^{-{\rm
i}\eta(a-y)} +\left(\frac{(\eta-\beta)p+{\rm
i}q}{2\eta}\right)e^{{\rm i}\eta(a-y)}.
\]
\end{lemm}

\begin{proof}
The general solution of the homogeneous second order equation \eqref{Afvp1} is
\[
  u(y)=c_1 e^{{\rm i}\eta y} + c_2 e^{-{\rm i}\eta y},
\]
 where $c_1$ and $c_2$ are constant coefficients to be determined. It follows
from the final conditions \eqref{Afvp2}--\eqref{Afvp3} that 
\[
u=p, \quad u'={\rm i}\beta p + q, \quad y=a.
\]
Plugging the final values of $u$ and $u'$ into the general solution, we obtain 
\[
 c_1 = \left(\frac{(\eta+\beta)p-{\rm i}q}{2\eta}\right)e^{-{\rm
i}\eta a},\quad
 c_2 = \left(\frac{(\eta-\beta)p+{\rm i}q}{2\eta}\right)e^{{\rm
i}\eta a},
\]
which completes the proof.
\end{proof}

Consider the two-point boundary value problem of the second order equation in
the interval $(0, h)$:
\begin{subequations}\label{Abvp}
\begin{align}
\label{Abvp1}  u'' + \beta^2 u  = v, \quad & 0 < y < h,\\
\label{Abvp2}  u' = r, \quad & y=0, \\
\label{Abvp3}  u' - {\rm i} \beta u = s, \quad & y =h,
\end{align}
\end{subequations}
where $0\neq\beta, r, s$ are constants.

\begin{lemm}\label{A2}
The two-point boundary value problem \eqref{Abvp} has a unique solution which
is given by 
\[
 u(y)=K_1(y; \beta)r-K_2(y; \beta)s+\int_0^h K_3(y, z; \beta)v(z){\rm d}z,
\]
where
\[
 K_1(y; \beta)=\frac{e^{{\rm i}\beta y}}{{\rm i}\beta},\quad
K_2(y; \beta)=\frac{e^{{\rm i}\beta h}}{2{\rm i}\beta}(e^{{\rm i}\beta
y}+e^{-{\rm i}\beta y}),
\]
and
\[
 K_3(y, z; \beta)=\begin{cases}
            \frac{e^{{\rm i}\beta y}}{2{\rm i}\beta}(e^{{\rm i}\beta z}+e^{-{\rm
i}\beta z}),\quad z<y,\\[5pt]
 \frac{e^{{\rm i}\beta z}}{2{\rm i}\beta}(e^{{\rm i}\beta y}+e^{-{\rm
i}\beta y}),\quad z>y.
           \end{cases}
\]
\end{lemm}

\begin{proof}
A fundamental set of solutions for the second order
equation \eqref{Abvp1} is 
 \[
  u_1(y)=e^{{\rm i}\beta y},\quad u_2(y)=e^{-{\rm i}\beta y}. 
 \]
A simple calculation yields that the Wronskian $W(u_1, u_2)=-2{\rm
i}\beta$. It follows from the variation of parameters that the general
solution to the nonhomogeneous second order equation \eqref{Abvp1} is
\begin{equation}\label{A2-s1}
 u(y)=c_1 e^{{\rm i}\beta y}  + c_2 e^{-{\rm i}\beta y} +\frac{e^{{\rm
i}\beta y}}{2{\rm i}\beta} \int_0^y e^{-{\rm i}\beta z} v(z){\rm d}z
 -\frac{e^{-{\rm i}\beta y}}{2{\rm i}\beta} \int_0^y e^{{\rm
i}\beta z} v(z){\rm d}z,
\end{equation}
where $c_1$ and $c_2$ are undetermined constants. 

Taking the derivative of \eqref{A2-s1}, evaluating at $y=0$, and using the
boundary condition \eqref{Abvp2} give
\begin{equation}\label{A2-s2}
 u'(0)={\rm i}\beta (c_1-c_2)=r.
\end{equation}
It follows from the boundary condition \eqref{Abvp3} that 
\begin{equation}\label{A2-s3}
 c_2=\frac{1}{2{\rm i}\beta}\left(\int_0^h e^{{\rm i}\beta z}v(z){\rm
d}z-se^{{\rm i}\beta h} \right).
\end{equation}
Combining \eqref{A2-s2} and \eqref{A2-s3} yields
\begin{equation}\label{A2-s4}
 c_1=c_2+\frac{r}{{\rm i}\beta}=\frac{1}{2{\rm i}\beta}\left(\int_0^h e^{{\rm
i}\beta z}v(z){\rm d}z-se^{{\rm i}\beta h} \right)+\frac{r}{{\rm i}\beta}.
\end{equation}
Substituting \eqref{A2-s3} and \eqref{A2-s4} into \eqref{A2-s1}, we obtain the
solution.
\end{proof}

\end{document}